\documentclass[12pt]{article}

\usepackage[T1]{fontenc}
\usepackage[utf8]{inputenc}
\usepackage{mathtools}
\usepackage{amsmath,amscd,diagbox}
\usepackage{indentfirst}
\usepackage{mathrsfs}
\usepackage{amsfonts}
\usepackage{arydshln}
\usepackage{enumerate}
\usepackage{setspace}
\usepackage{amssymb,amsthm,cases}
\usepackage{amsbsy,pifont}
\usepackage{latexsym,tabularx}
\usepackage{cite}
\usepackage{bbm,multirow}
\usepackage{graphicx,epsfig,extarrows,dsfont}
\usepackage[final,allcolors=blue,colorlinks=true]{hyperref}
\usepackage[normalem]{ulem}
\input xy
\usepackage{caption,subcaption} 
\usepackage{tikz}
\usepackage{tikz-cd}
\usetikzlibrary{decorations.pathreplacing,decorations.markings}
\usetikzlibrary{cd}
\usetikzlibrary{shapes.geometric}
\usetikzlibrary{arrows,arrows.meta}
\usetikzlibrary{graphs,positioning}
\usetikzlibrary{matrix,decorations.pathmorphing}
\usetikzlibrary{math,calc,intersections,through,angles,arrows.meta,shapes.geometric,shadows,quotes,spy,datavisualization,datavisualization.formats.functions,plotmarks}
\tikzset{every picture/.style={samples=300,smooth,line join=round,thick,>=stealth}}

\tikzset{
	on each segment/.style={
		decorate,
		decoration={
			show path construction,
			moveto code={},
			lineto code={
				\path [#1]
				(\tikzinputsegmentfirst) -- (\tikzinputsegmentlast);
			},
			curveto code={
				\path [#1] (\tikzinputsegmentfirst)
				.. controls
				(\tikzinputsegmentsupporta) and (\tikzinputsegmentsupportb)
				..
				(\tikzinputsegmentlast);
			},
			closepath code={
				\path [#1]
				(\tikzinputsegmentfirst) -- (\tikzinputsegmentlast);
			},
		},
	},
	mid arrow/.style={postaction={decorate,decoration={
				markings,
				mark=at position .5 with {\arrow[#1]{stealth}}
	}}},
}

\numberwithin{equation}{section}

\theoremstyle{plain}
\newtheorem{theorem}{Theorem}[section]

\newtheorem{corollary}[theorem]{Corollary}

\newtheorem{lemma}[theorem]{Lemma}
\newtheorem{proposition}[theorem]{Proposition}

\theoremstyle{definition}
\newtheorem{definition}[theorem]{Definition}

\newtheorem{remark}[theorem]{Remark}

\def\XXint#1#2#3{{\setbox0=\hbox{$#1{#2#3}{\int}$}
		\vcenter{\hbox{$#2#3$}}\kern-.5\wd0}}

\DeclareMathSymbol{\subseteq}{\mathrel}{symbols}{"12}
\DeclareMathSymbol{\supseteq}{\mathrel}{symbols}{"13} 
\DeclareMathSymbol{\subsetneq}{\mathrel}{AMSb}{"28}                  \DeclareMathSymbol{\supsetneq}{\mathrel}{AMSb}{"29}    

\DeclareMathSymbol{\nsubseteq}{\mathrel}{AMSb}{"2A}                  

\DeclareMathSymbol{\nsupseteq}{\mathrel}{AMSb}{"2B}

\DeclareMathDelimiter{\langle}{\mathop}{symbols}{"68}{largesymbols}{"0A}
\DeclareMathDelimiter{\rangle}{\mathclose}{symbols}{"69}{largesymbols}{"0B}

\DeclareSymbolFont{txfontsA}{U}{txmia}{m}{it}
\SetSymbolFont{txfontsA}{bold}{U}{txmia}{bx}{it}
\DeclareFontSubstitution{U}{txmia}{m}{it}
\DeclareMathSymbol{\upalpha}{\mathord}{txfontsA}{"0B}
\DeclareMathSymbol{\upbeta}{\mathord}{txfontsA}{"0C}
\DeclareMathSymbol{\upgamma}{\mathord}{txfontsA}{"0D}
\DeclareMathSymbol{\updelta}{\mathord}{txfontsA}{"0E}
\DeclareMathSymbol{\upepsilon}{\mathord}{txfontsA}{"0F}
\DeclareMathSymbol{\upzeta}{\mathord}{txfontsA}{"10}
\DeclareMathSymbol{\upeta}{\mathord}{txfontsA}{"11}
\DeclareMathSymbol{\uptheta}{\mathord}{txfontsA}{"12}
\DeclareMathSymbol{\upiota}{\mathord}{txfontsA}{"13}
\DeclareMathSymbol{\upkappa}{\mathord}{txfontsA}{"14}
\DeclareMathSymbol{\uplambda}{\mathord}{txfontsA}{"15}
\DeclareMathSymbol{\upmu}{\mathord}{txfontsA}{"16}
\DeclareMathSymbol{\upnu}{\mathord}{txfontsA}{"17}
\DeclareMathSymbol{\upxi}{\mathord}{txfontsA}{"18}
\DeclareMathSymbol{\uppi}{\mathord}{txfontsA}{"19}
\DeclareMathSymbol{\uprho}{\mathord}{txfontsA}{"1A}
\DeclareMathSymbol{\upsigma}{\mathord}{txfontsA}{"1B}
\DeclareMathSymbol{\uptau}{\mathord}{txfontsA}{"1C}
\DeclareMathSymbol{\upupsilon}{\mathord}{txfontsA}{"1D}
\DeclareMathSymbol{\upphi}{\mathord}{txfontsA}{"1E}
\DeclareMathSymbol{\upchi}{\mathord}{txfontsA}{"1F}
\DeclareMathSymbol{\uppsi}{\mathord}{txfontsA}{"20}
\DeclareMathSymbol{\upomega}{\mathord}{txfontsA}{"21}
\DeclareMathSymbol{\upvarepsilon}{\mathord}{txfontsA}{"22}
\DeclareMathSymbol{\upvartheta}{\mathord}{txfontsA}{"23}
\DeclareMathSymbol{\upvarpi}{\mathord}{txfontsA}{"24}
\DeclareMathSymbol{\upvarrho}{\mathord}{txfontsA}{"25}
\DeclareMathSymbol{\upvarsigma}{\mathord}{txfontsA}{"26}
\DeclareMathSymbol{\upvarphi}{\mathord}{txfontsA}{"27}                   

\DeclareSymbolFont{ugmL}{OMX}{mdugm}{m}{n}
\SetSymbolFont{ugmL}{bold}{OMX}{mdugm}{b}{n}
\DeclareMathAccent{\wideparen}{\mathord}{ugmL}{"F3}

\def\pt{\partial}

\def\ra{\rightarrow}

\def\e{\epsilon}

\def\ol{\overline}

\def\vp{\varphi}
\def\lg{\langle}

\def\rg{\rangle}

\def\fa{\forall}

\def\bf{\textbf}
\def\pt{\partial}

\def\om{\omega}
\def\Om{\Omega}
\def\la{\lambda}
\def\al{\alpha}
\def\be{\beta}
\def\de{\delta}
\def\ga{\gamma}

\def\si{\sigma}

\def\Ga{\Gamma}

\def\ts{\times}

\def\iy{\infty}

\def\f{\frac}

\def\Lra{\Leftrightarrow}

\def\Span{{\rm{span}}}

\def\df{\mathrm d}

\def\wh{\widehat}

\def\esssup{\operatorname*{ess\ \! sup}}

\def\diag{{\rm diag}}

\def\mcA{\mathcal{A}}

	\DeclareMathOperator{\Div}{div}

	\newcommand{\R}{\mathbb R}
	\newcommand{\N}{\mathbb N}

\begin{document}

\title{Hidden Boundary Trace Regularity and an Observability Estimate with Interior Remainder for Boundary-Degenerate Hyperbolic Equations}

\author{Dong-Hui Yang, J. Zhong}

\date{}

\maketitle{}

\begin{abstract}
	We study hidden boundary trace regularity for two-dimensional hyperbolic equations with boundary degeneracy governed by $\mcA\vp=-\Div(A\nabla \vp)$, where $A=\diag(1,r^\al)$ and $\al\in(0,1)$. We establish well-posedness in weighted Sobolev spaces and prove an $L^2$ trace estimate for the normal derivative on the nondegenerate side $r=1$. Using truncated geometries and Carleman weights adapted to the anisotropic degeneracy, we derive a large-time observability estimate with a lower-order interior remainder. We also identify a framework-level obstruction at the critical threshold $\al=1$: the weighted Dirichlet coercivity underlying the subcritical analysis loses uniformity and exhibits a logarithmic loss on truncated domains.

		\end{abstract}



\section{Introduction}

We consider the following two-dimensional degenerate hyperbolic equation:
	\begin{equation}\label{01.10.1}
		\begin{cases}
			\pt_{tt}\vp-\Div (A\nabla \vp)= f, &\mbox{in } Q,\\
			\vp=0, &\mbox{on }\pt Q, \\
			\vp(0)=\vp^0,\pt_t\vp(0)=\vp^1, &\mbox{in } \Om,
		\end{cases}
	\end{equation}
where $\Om=(0,1)\ts (0,1)$, $Q=\Om\ts (0,T)$ for some $T>0$, $A=\diag(1,r^{\al})$ with $\al\in (0,1)$, $\vp^0\in H_0^1(\Om;w)$, $\vp^1\in L^2(\Om)$, and $f\in L^2(Q)$. The weighted Sobolev space $H_0^1(\Om;w)$ will be defined in Section \ref{S2}. 

We write $z=(\theta,r)\in\Om$ and
\begin{equation*}
	w=r^\al, \quad A=\diag(1,w),\mbox{ and }  \mcA \vp=-\Div(A\nabla \vp). 
\end{equation*} 
We also set $\Ga_2^0=\{z=(\theta,r)\in\ol\Om\colon r=0\}$ and $\Ga=\pt\Om-\Ga_2^0$. For $\de\in [0,1]$, we denote by $\Om_\de=\{z=(\theta,r)\in\Om\colon r>\de\}$, $\Ga_2^\de=\{z\in\ol\Om_\de\colon r=\de\}$, and $\Ga_\de=\pt\Om_\de-\Ga_2^\de$. We also write
\begin{equation*}
	\Om_\de^\sharp=(\de,1-\de)\ts (\de,1),
\end{equation*}
which is the portion of the domain bounded away from the four corner points.  

The main difficulty comes from the fact that the equation is hyperbolic, while the diffusion matrix is anisotropic and degenerates on the boundary $\Ga_2^0=\{r=0\}$. This degeneracy prevents one from applying directly the standard hidden regularity theory for strictly hyperbolic equations; see, for example, \cite{Lions,Lasiecka} and the monographs \cite{Weiss,Zuazua}. In particular, the normal trace on the nondegenerate part of the boundary cannot be handled by the usual energy argument alone, and the weighted structure of the operator must be incorporated into both the functional setting and the Carleman weight.

Problems with degeneracy have been intensively studied in the parabolic setting, especially through Carleman estimates and controllability arguments; see \cite{AlabauPar,Araruna,Cannarsa1}. For general background on control-theoretic and Carleman-based methods, we also refer to \cite{Fursikov,Coron,Lebeau}. By comparison, the literature on degenerate hyperbolic equations is much more limited. Representative one-dimensional controllability results may be found in \cite{AlabauWave,Gueye,Fragnelli,ZhangGao}, whereas more anisotropic or higher-dimensional models are treated in \cite{YangGrushin,Yang}. In higher-dimensional boundary-degenerate configurations, both the functional setting and the boundary analysis become substantially more delicate.

The terminology ``shape design'' is used here in the sense of the approximation strategy developed in \cite{GuoShape,YangShape}: one replaces the original degenerate geometry by a family of nondegenerate or less singular auxiliary configurations and chooses the localization together with the Carleman weight so that the resulting estimates remain compatible with the anisotropic geometry of the limit problem.

The paper may be viewed as having two closely related parts. The first one concerns hidden regularity. We formulate \eqref{01.10.1} in weighted Sobolev spaces naturally associated with the coefficient $r^\al$. In Section \ref{S2}, we introduce the spaces $H^1(\Om;w)$, $H_0^1(\Om;w)$, and $D(\mcA)$, prove a Hardy-type inequality, and collect the basic well-posedness and regularity properties of weak solutions. In particular, the solution enjoys enough weighted interior regularity to justify the integrations by parts that will be used in the Carleman analysis, and we obtain the boundary trace estimate
\begin{equation*}
	\|\pt_r\vp\|_{L^2(0,T;L^2(\Ga_2^1))}\leq C\Big(\|\vp^0\|_{H_0^1(\Om;w)}+\|\vp^1\|_{L^2(\Om)}+\|f\|_{L^2(Q)}\Big),
\end{equation*}
which is the qualitative hidden boundary trace regularity part of the paper.

The second part is the shape-design component. In the background, one keeps track of the truncated geometries $\Om_\de=\{r>\de\}$ and of the limiting process $\Om_\de\to\Om$ as $\de\to 0^+$, while the localization and the Carleman weight are designed so as to remain compatible with the anisotropic geometry of the degenerate boundary. The core of the paper is then a Carleman estimate established after localizing the solution away from the lateral sides $\theta=0$ and $\theta=1$. The weight
\begin{equation*}
	\xi(\theta,r,t)=\theta^2+r^{2-\al}-\be(t-t_0)^2,
\end{equation*}
is tailored to the anisotropic scaling of the operator, so that the radial part is compatible with the degeneracy $r^\al$. After conjugation with $e^{s\si}$, the resulting bulk terms and boundary terms can be arranged to produce a coercive estimate. A key point is that the boundary contribution on $\Ga_2^1=\{r=1\}$ survives with the correct sign, while the degenerate boundary $\Ga_2^0=\{r=0\}$ is absorbed by weighted trace identities.

As a consequence, for $T$ satisfying the sufficient condition in Remark \ref{04.05.R2}, we derive the observability estimate with interior remainder stated in Theorems \ref{04.05.T1} and \ref{04.05.T2}: the initial energy is bounded by the $L^2$-norm of $\pt_r\vp$ on the portion $(\de_0,1-\de_0)\times\{1\}$ of the top boundary, that is, away from the corner neighborhoods, plus a lower-order interior term supported near the vertical sides. Thus, Section \ref{S2} supplies the qualitative hidden boundary trace estimate needed to make sense of the observation, whereas Section \ref{S3} yields a quantitative large-time estimate linking the initial energy to that observation. This estimate can also be interpreted as a quantitative hidden regularity statement for the normal derivative on the nondegenerate boundary, but it is not a pure observability inequality because of the interior remainder.

At the same time, the subcritical character of the weighted Dirichlet framework raises a natural question about the critical threshold $\al=1$. In the last section we address this issue at the level of the functional setting. More precisely, we show that the coercive mechanism supporting the present framework cannot be continued uniformly to the endpoint: on truncated domains, the corresponding critical Hardy-Poincar\'e constant necessarily blows up at a logarithmic rate. This does not constitute an endpoint theory at $\al=1$, but it clarifies the framework-level obstruction encountered by the present analysis.

The paper is organized as follows. In Section \ref{S2} we establish the weighted functional framework and the preliminary regularity results. Section \ref{S3} is devoted to the derivation of the Carleman estimate and to the proof of the final observability estimate with interior remainder. Section \ref{S4} discusses the framework-level obstruction of the subcritical Dirichlet setting at the critical threshold $\al=1$.

\section{Preliminaries}\label{S2}
 
We define the weighted Sobolev space
\begin{equation*}
	H^1(\Om;w)=\left\{u\in L^2(\Om)\colon \int_\Om A\nabla u\cdot \nabla u\df z<+\iy\right\},
\end{equation*}
equipped with the inner product and norm
\begin{equation*}
	(u,v)_{H^1(\Om;w)}=\int_\Om \left(uv+\nabla u\cdot A\nabla v\right)\df z, \quad \|u\|_{H^1(\Om;w)}=(u,u)_{H^1(\Om;w)}^\f{1}{2}. 
\end{equation*}
Set
\begin{equation*}
	H_0^1(\Om;w)=\mbox{the closure of $C_0^\iy(\Om)$ in } H^1(\Om;w). 
\end{equation*}
This is the weighted Dirichlet space used throughout the paper. In particular, the zero boundary condition is imposed on the whole boundary $\pt\Om$, including the degenerate side $\Ga_2^0=\{r=0\}$. This is consistent with the weakly degenerate regime $\al\in (0,1)$ in the standard terminology; see, for instance, \cite{Cannarsa1}.

We also define
\begin{equation*}
	H^2(\Om;w)=\left\{u\in H^1(\Om;w)\colon \mcA u\in L^2(\Om)\right\}, 
\end{equation*}
equipped with the inner product and norm
\begin{equation*}
	(u,v)_{H^2(\Om;w)}=(u,v)_{H^1(\Om;w)}+\int_\Om (\mcA u)(\mcA v)\df z,\quad \|u\|_{H^2(\Om;w)}=(u,u)_{H^2(\Om;w)}^\f{1}{2}. 
\end{equation*}

The spaces 
\begin{equation*}
	H_0^1(\Om;w),\quad H^1(\Om;w),\quad H^2(\Om;w)
\end{equation*}
are Hilbert spaces; see \cite{GC,Heinonen}.

Finally, set
\begin{equation*}
	D(\mcA)=H^2(\Om;w)\cap H_0^1(\Om;w). 
\end{equation*}

\begin{lemma}\label{03.11.L1}
	Let $\al\in (0,1)$. Then for all $u\in H_0^1(\Om;w)$ we have
	\begin{equation*}
		\int_\Om r^{\al-2}u^2\df z\leq \f{4}{(1-\al)^2}\int_\Om r^\al (\pt_{r}u)^2\df z. 
	\end{equation*}
\end{lemma}

\begin{proof}
	Since $C_0^\iy(\Om)$ is dense in $H_0^1(\Om;w)$, it is enough to prove the estimate for $u\in C_0^\iy(\Om)$. For each fixed $\theta\in (0,1)$, using $u(\theta,0)=0$, we write
	\[
		u(\theta,r)=\int_0^r \pt_r u(\theta,s)\df s.
	\]
	Here the identity is used only for smooth approximants. The statement for general $u\in H_0^1(\Om;w)$ then follows by density, so no additional pointwise trace property is required at this stage.
	By Cauchy-Schwarz inequality, for any $\be\in (\al,1)$ we have
	\[
		|u(\theta,r)|^2\leq \left(\int_0^r s^\be(\pt_r u(\theta,s))^2\df s\right)\left(\int_0^r s^{-\be}\df s\right)
		=\f{1}{1-\be}\left(\int_0^r s^\be(\pt_r u(\theta,s))^2\df s\right)r^{1-\be}.
	\]
	Multiplying by $r^{\al-2}$ and integrating in $r$ and $\theta$, we obtain
	\begin{equation*}
		\begin{split}
			\int_\Om r^{\al-2}u^2\df z
			&\leq \f{1}{1-\be}\int_0^1\int_0^1 r^{\al-\be-1}\left(\int_0^r s^\be(\pt_r u(\theta,s))^2\df s\right)\df r\df \theta\\
			&=\f{1}{1-\be}\int_0^1\int_0^1\left(\int_s^1 r^{\al-\be-1}\df r\right)s^\be(\pt_r u(\theta,s))^2\df s\df \theta\\
			&\leq \f{1}{(1-\be)(\be-\al)}\int_\Om r^\al(\pt_r u)^2\df z.
		\end{split}
	\end{equation*}
	Choosing $\be=\f{1+\al}{2}$ gives the constant $\f{4}{(1-\al)^2}$. The conclusion for general $u\in H_0^1(\Om;w)$ follows by density.
\end{proof}

\begin{remark}\label{03.11.R1}
	Let $\al\in (0,1)$. Then for all $u\in H_0^1(\Om;w)$ we have 
	\begin{equation}\label{03.11.1}
		\int_\Om u^2\df z=\int_\Om r^{2-\al}r^{\al-2}u^2\df z\leq \f{4}{(1-\al)^2}\int_\Om r^\al (\pt_{r}u)^2\df z. 
	\end{equation}
	Hence
	\begin{equation*}
		(u,v)_{H_0^1(\Om;w)}=\int_\Om A\nabla u\cdot \nabla v\df z,\quad \|u\|_{H_0^1(\Om;w)}=(u,u)_{H_0^1(\Om;w)}^\f{1}{2}
	\end{equation*}
	defines an inner product and an equivalent norm on $H_0^1(\Om;w)$. 
\end{remark}

\begin{lemma}\label{03.11.L2}
	Let $\al\in (0,1)$. Then the embedding
	\[
		H_0^1(\Om;w)\hookrightarrow L^2(\Om)
	\]
	is compact.
\end{lemma}

\begin{proof}
	Let $\{u_n\}_{n\in\N}$ be bounded in $H_0^1(\Om;w)$. Since $H_0^1(\Om;w)$ is a Hilbert space, after passing to a subsequence we may assume
	\[
		u_n\rightharpoonup u \quad \mbox{weakly in } H_0^1(\Om;w).
	\]
	It is enough to prove that $u_n\to u$ strongly in $L^2(\Om)$; equivalently, replacing $u_n$ by $u_n-u$, we may assume $u=0$ and show $u_n\to 0$ in $L^2(\Om)$.
	
	Let $\e>0$. For $\de\in (0,\f18)$, since $\Om-\Om_\de=(0,1)\ts (0,\de)$, Lemma \ref{03.11.L1} implies
	\begin{equation*}
		\begin{split}
			\int_{\Om-\Om_\de}u_n^2\df z
			&=\int_{\Om-\Om_\de}r^{2-\al}r^{\al-2}u_n^2\df z
			\leq \de^{2-\al}\int_\Om r^{\al-2}u_n^2\df z\\
			&\leq C\de^{2-\al}\|u_n\|_{H_0^1(\Om;w)}^2.
		\end{split}
	\end{equation*}
	Since $\{u_n\}$ is bounded in $H_0^1(\Om;w)$, we may choose $\de_0\in (0,\f18)$ so small that
	\[
		\int_{\Om-\Om_{\de_0}}u_n^2\df z<\f{\e^2}{4}
	\]
	for all $n\in\N$.
	
	On the truncated domain $\Om_{\de_0}$, the weight $r^\al$ is bounded above and below by positive constants. Hence the norms of $H^1(\Om_{\de_0})$ and $H^1(\Om_{\de_0};w)$ are equivalent there. In particular, $\{u_n|_{\Om_{\de_0}}\}$ is bounded in $H^1(\Om_{\de_0})$, and since $u_n\rightharpoonup 0$ weakly in $H_0^1(\Om;w)$, its restriction converges weakly to $0$ in $H^1(\Om_{\de_0})$. By the Rellich theorem, after passing to a further subsequence if necessary, we have
	\[
		u_n\to 0 \quad \mbox{strongly in } L^2(\Om_{\de_0}).
	\]
	Therefore there exists $n_\e\in\N$ such that
	\[
		\int_{\Om_{\de_0}}u_{n_\e}^2\df z<\f{\e^2}{4}.
	\]
	Combining the estimates on $\Om_{\de_0}$ and $\Om-\Om_{\de_0}$, we get
	\[
		\|u_{n_\e}\|_{L^2(\Om)}<\e.
	\]
	This proves that every bounded sequence in $H_0^1(\Om;w)$ admits a subsequence converging strongly in $L^2(\Om)$.
\end{proof}

\begin{definition}\label{03.11.D1}
	Let $\vp^0\in H_0^1(\Om;w)$, $\vp^1\in L^2(\Om)$ and $f\in L^2(Q)$. We call
	\begin{equation*}
		\vp\in L^2(0,T; H_0^1(\Om;w))\cap H^1(0,T; L^2(\Om))\cap H^2(0,T; H^{-1}(\Om;w))
	\end{equation*}
	a weak solution of \eqref{01.10.1} with respect to $(\vp^0,\vp^1,f)$ if
	
	(i) for each $v\in H_0^1(\Om;w)$ and a.e.~$t\in [0,T]$, we have
	\begin{equation*}
		\lg\pt_{tt}\vp, v\rg_{H^{-1}(\Om;w), H_0^1(\Om;w)}+(\vp, v)_{H_0^1(\Om;w)}=(f,v)_{L^2(\Om)}, 
	\end{equation*}
	
	(ii) $\vp(0)=\vp^0$ and $\pt_t\vp(0)=\vp^1$. 
\end{definition}

\begin{theorem}\label{03.11.T1}
	Let $\vp^0\in H_0^1(\Om;w)$, $\vp^1\in L^2(\Om)$ and $f\in L^2(Q)$. Then problem \eqref{01.10.1} admits a unique weak solution
	\begin{equation*}
		\vp\in L^2(0,T; H_0^1(\Om;w))\cap H^1(0,T; L^2(\Om))\cap H^2(0,T; H^{-1}(\Om;w))
	\end{equation*}
	with data $(\vp^0,\vp^1, f)$, and
		\begin{equation}\label{03.12.1}
			\begin{split}
				&\esssup_{t\in [0,T]}\left(\|\vp(t)\|_{H_0^1(\Om;w)}+\|\pt_t\vp(t)\|_{L^2(\Om)}\right)+\|\pt_{tt}\vp\|_{L^2(0,T; H^{-1}(\Om;w))}\\
				&\hspace{4.5mm}+\left\|\pt_r\vp\right\|_{L^2(0,T; L^2(\Ga_2^1))}\\
				&\leq C\left(\|\vp^0\|_{H_0^1(\Om;w)}+\|\vp^1\|_{L^2(\Om)}+\|f\|_{L^2(Q)}\right), 
			\end{split}
	\end{equation}
	where the positive constant $C$ depends only on $\al$ and $T$. 
	
	Furthermore, if $\vp^0\in D(\mcA)$, $\vp^1\in H_0^1(\Om;w)$, and $f\in H^1(0,T; L^2(\Om))$, then 
		\begin{equation}\label{03.12.2}
			\begin{split}
				&\esssup_{t\in [0,T]}\left(\|\vp(t)\|_{H^2(\Om_\f{1}{8}^\sharp)}+\|\vp(t)\|_{D(\mcA)}+\|\pt_t\vp(t)\|_{H_0^1(\Om;w)}+\|\pt_{tt}\vp(t)\|_{L^2(\Om)}\right)\\
				&\leq C\left(\|\vp^0\|_{D(\mcA)}+\|\vp^1\|_{H_0^1(\Om;w)}+\|f\|_{H^1(0,T; L^2(\Om))}\right), 
			\end{split}
		\end{equation}
	where the positive constant $C$ depends only on $\al$ and $T$. 
\end{theorem}

\begin{proof}
	We sketch the standard Galerkin construction and record the additional boundary trace estimate needed later.
	
	{\it Step 1: Galerkin approximation and weak solvability.}
	
	Define
	\[
		a(u,v)=\int_\Om A\nabla u\cdot \nabla v\df z,\qquad u,v\in H_0^1(\Om;w).
	\]
		By Remark \ref{03.11.R1}, $a(\cdot,\cdot)$ is continuous and coercive on $H_0^1(\Om;w)$. By Lemma \ref{03.11.L2}, the embedding $H_0^1(\Om;w)\hookrightarrow L^2(\Om)$ is compact. Therefore the spectral theorem for compactly embedded coercive forms yields a sequence of eigenpairs $\{(\la_k,e_k)\}_{k\in\N}\subset (0,+\iy)\times D(\mcA)$ such that
	\[
		\mcA e_k=\la_k e_k,\qquad (e_k,e_j)_{L^2(\Om)}=\de_{kj},
	\]
	and $\{e_k\}_{k\in\N}$ is an orthonormal basis of $L^2(\Om)$. In particular,
	\[
		a(e_k,e_j)=\la_k\de_{kj}.
	\]
	For each $n\in\N$, we seek an approximate solution of the form
	\[
		\vp_n(t)=\sum_{k=1}^n d_k^n(t)e_k,
	\]
	where the coefficients solve
	\[
		(d_k^n)''(t)+\la_k d_k^n(t)=(f(t),e_k)_{L^2(\Om)},\qquad d_k^n(0)=(\vp^0,e_k)_{L^2(\Om)},\qquad (d_k^n)'(0)=(\vp^1,e_k)_{L^2(\Om)}.
	\]
	This finite-dimensional system has a unique solution on $[0,T]$. Multiplying the $k$-th equation by $(d_k^n)'(t)$, summing over $k=1,\ldots,n$, and integrating in time, we obtain
	\begin{equation*}
		\esssup_{t\in [0,T]}\left(\|\vp_n(t)\|_{H_0^1(\Om;w)}^2+\|\pt_t\vp_n(t)\|_{L^2(\Om)}^2\right)
		\leq C\left(\|\vp^0\|_{H_0^1(\Om;w)}^2+\|\vp^1\|_{L^2(\Om)}^2+\|f\|_{L^2(Q)}^2\right).
	\end{equation*}
	Moreover, from the Galerkin equation
	\[
		\langle \pt_{tt}\vp_n,v\rangle_{H^{-1}(\Om;w),H_0^1(\Om;w)}+a(\vp_n,v)=(f,v)_{L^2(\Om)}\qquad \forall v\in \Span\{e_1,\ldots,e_n\},
	\]
	we infer
	\begin{equation*}
		\|\pt_{tt}\vp_n\|_{L^2(0,T;H^{-1}(\Om;w))}
		\leq C\left(\|\vp^0\|_{H_0^1(\Om;w)}+\|\vp^1\|_{L^2(\Om)}+\|f\|_{L^2(Q)}\right).
	\end{equation*}
	Hence, after extracting a subsequence if necessary, we have
	\[
		\vp_n\rightharpoonup \vp \mbox{ weakly-* in }L^\iy(0,T;H_0^1(\Om;w)),\quad \pt_t\vp_n\rightharpoonup \pt_t\vp \mbox{ weakly-* in }L^\iy(0,T;L^2(\Om)),
	\]
	and
	\[
		\pt_{tt}\vp_n\rightharpoonup \pt_{tt}\vp \mbox{ weakly in }L^2(0,T;H^{-1}(\Om;w)).
	\]
		Passing to the limit in the Galerkin formulation shows that $\vp$ is a weak solution of \eqref{01.10.1} in the sense of Definition \ref{03.11.D1}. Uniqueness follows from the same energy identity applied to the difference of two weak solutions. Therefore the first three terms in \eqref{03.12.1} satisfy the stated estimate, and
	\[
		\vp\in L^2(0,T; H_0^1(\Om;w))\cap H^1(0,T; L^2(\Om))\cap H^2(0,T; H^{-1}(\Om;w)).
	\]
	
	{\it Step 2: Higher regularity.}
	
	Assume now that $\vp^0\in D(\mcA)$, $\vp^1\in H_0^1(\Om;w)$ and $f\in H^1(0,T; L^2(\Om))$. Differentiating \eqref{01.10.1} with respect to $t$, we obtain that $\Psi=\pt_t\vp$ satisfies the same hyperbolic equation with source $\pt_t f$ and initial data
	\[
		\Psi(0)=\vp^1,\qquad \pt_t\Psi(0)=f(0)-\mcA \vp^0\in L^2(\Om).
	\]
	Applying Step 1 to $\Psi$, we get
	\begin{equation*}
		\begin{split}
			&\esssup_{t\in [0,T]}\|\pt_t\vp(t)\|_{H_0^1(\Om;w)}
			+\esssup_{t\in [0,T]}\|\pt_{tt}\vp(t)\|_{L^2(\Om)}\\
			&\leq C\left(\|\vp^0\|_{D(\mcA)}+\|\vp^1\|_{H_0^1(\Om;w)}
			+\|f\|_{H^1(0,T;L^2(\Om))}\right).
		\end{split}
	\end{equation*}
	Consequently,
	\[
		\mcA\vp=\pt_{tt}\vp-f\in L^\iy(0,T; L^2(\Om)),
	\]
		and therefore $\vp\in L^\iy(0,T; D(\mcA))$. Away from the degenerate boundary $r=0$ and away from the four corner points, the operator is uniformly elliptic and the boundary is locally smooth. Hence a standard cut-off argument combined with interior and local boundary elliptic regularity yields
		\[
			\vp\in L^\iy(0,T; H^2(\Om_{\f{1}{8}}^\sharp)).
		\]
			This establishes \eqref{03.12.2}.
	
	{\it Step 3: Boundary trace estimate on $\Ga_2^1$.}
	
	We first consider smooth data and the associated Galerkin approximations. Let $\rho\in C^\iy([0,1])$ satisfy $\rho=0$ near $r=0$ and $\rho=1$ near $r=1$. Multiplying the regularized equation by $\rho(r)\pt_r\vp$ and integrating over $Q$, the boundary terms on $\{\theta=0\}\cup\{\theta=1\}$ vanish because the multiplier is purely radial, while the contributions near $r=0$ disappear since $\rho$ vanishes there. After integrating by parts, using the energy bounds from Step 1 and estimate \eqref{03.11.1} to control the lower-order terms, we obtain
	\[
		\|\pt_r\vp\|_{L^2(0,T;L^2(\Ga_2^1))}\leq C\Big(\|\vp^0\|_{H_0^1(\Om;w)}+\|\vp^1\|_{L^2(\Om)}+\|f\|_{L^2(Q)}\Big),
	\]
	with $C$ independent of the Galerkin dimension. Passing to the limit by density gives the fourth term in \eqref{03.12.1} for general data.
	
	Combining Steps 1--3 completes the proof.
\end{proof}
 
\begin{lemma}[Stationary weighted elliptic estimate]\label{02.13.L2}
	Let $\al\in(0,1)$. There exists a constant $C>0$, depending only on $\al$, such that for every $u\in D(\mcA)$,
	\begin{equation}\label{02.13.1}
		\begin{split}
			&\|\pt_{\theta\theta}u\|_{L^2(\Om)}^2+\|r^{\al/2}\pt_{\theta r}u\|_{L^2(\Om)}^2+\|r^{1+\al/2}\pt_{rr}u\|_{L^2(\Om)}^2\\
			&\hspace{12mm}+\|\pt_r(r^\al\pt_r u)\|_{L^2(\Om)}^2
			\leq C\|\mcA u\|_{L^2(\Om)}^2 .
		\end{split}
	\end{equation}
\end{lemma}

\begin{proof}
	We give the argument in order to make explicit why the estimate is compatible with the corners of the rectangle. By density, it is enough to consider smooth functions satisfying the homogeneous Dirichlet condition. Let $g=\mcA u$ and expand in the sine basis in the $\theta$-variable:
	\[
		u(\theta,r)=\sum_{n\geq 1}u_n(r)e_n(\theta),\qquad 
		g(\theta,r)=\sum_{n\geq 1}g_n(r)e_n(\theta),\qquad e_n(\theta)=\sqrt2\sin(n\pi\theta).
	\]
	With $\mu_n=(n\pi)^2$, each coefficient satisfies
	\[
		g_n=\mu_n u_n-\pt_r(r^\al u_n').
	\]
	Multiplying by $u_n$ and integrating by parts in $r$ gives
	\[
		\mu_n\|u_n\|_{L^2(0,1)}^2+\int_0^1r^\al |u_n'|^2\,dr
		=\int_0^1 g_nu_n\,dr .
	\]
	Moreover, another integration by parts yields the exact identity
	\[
		\|g_n\|_{L^2(0,1)}^2
		=\mu_n^2\|u_n\|_{L^2(0,1)}^2+\|\pt_r(r^\al u_n')\|_{L^2(0,1)}^2
		+2\mu_n\int_0^1r^\al |u_n'|^2\,dr .
	\]
	Summing over $n$ gives the bounds for $\pt_{\theta\theta}u$, $r^{\al/2}\pt_{\theta r}u$, and $\pt_r(r^\al\pt_r u)$. Finally,
	\[
		r^{1+\al/2}u_n''=r^{1-\al/2}\pt_r(r^\al u_n')-\al r^{\al/2}u_n',
	\]
	and $r^{1-\al/2}\leq 1$ on $(0,1)$. Hence the already obtained estimates also control $r^{1+\al/2}u_{rr}$. This proves \eqref{02.13.1}.
\end{proof}

\begin{corollary}[Global weighted second-order regularity]\label{02.14.C1}
	Let $\vp^0\in D(\mcA)$, $\vp^1\in H_0^1(\Om;w)$, and $f\in H^1(0,T;L^2(\Om))$. Let $\vp$ denote the solution of \eqref{01.10.1} associated with $(\vp^0,\vp^1,f)$. Then 
	\begin{equation}\label{02.15.2}
		\pt_{\theta \theta }\vp\in L^2(Q), \  r^\f{\al}{2}\pt_{\theta r}\vp\in L^2(Q), \  r^{1+\f{\al}{2}}\pt_{rr}\vp\in L^2(Q), \mbox{ and }\pt_r(r^\al\pt_r\vp)\in L^2(Q).
	\end{equation}
	Moreover, there exists a positive constant $C$, depending only on $\al$, such that 
	\begin{equation}\label{02.15.1}
		\begin{split} 
			&\iint_Q \left[(\pt_{\theta\theta}\vp)^2+r^\al (\pt_{\theta r}\vp)^2+r^{2+\al}(\pt_{rr}\vp)^2\right]\df z\df t\\
			&\leq C\left(\|\vp^0\|_{D(\mcA)}^2+\|\vp^1\|_{H_0^1(\Om;w)}^2+\|f\|_{H^1(0,T; L^2(\Om))}^2\right). 
		\end{split} 
	\end{equation}
\end{corollary}

\begin{proof}
	By the second part of Theorem \ref{03.11.T1}, we have
	\[
		\vp\in L^\iy(0,T; D(\mcA))\cap H^1(0,T; H_0^1(\Om;w))\cap H^2(0,T; L^2(\Om)).
	\]
	Hence, for a.e.~$t\in (0,T)$,
	\[
		\mcA\vp(t)=\pt_{tt}\vp(t)-f(t)\in L^2(\Om).
	\]
	The stationary weighted elliptic estimate in Lemma \ref{02.13.L2} yields, for a.e.~$t\in (0,T)$,
	\begin{equation*}
		\begin{split}
			&\|\pt_{\theta\theta}\vp(t)\|_{L^2(\Om)}^2+\|r^\f{\al}{2}\pt_{\theta r}\vp(t)\|_{L^2(\Om)}^2+\|r^{1+\f{\al}{2}}\pt_{rr}\vp(t)\|_{L^2(\Om)}^2\\
			&\hspace{4.5mm}+\|\pt_r(r^\al\pt_r\vp(t))\|_{L^2(\Om)}^2
			\leq C\|\mcA\vp(t)\|_{L^2(\Om)}^2,
			\end{split}
		\end{equation*}
		where $C$ depends only on $\al$. This is a global weighted estimate and should be distinguished from the classical $H^2$-regularity statement in \eqref{03.12.2}, which is only claimed on $\Om_{1/8}^\sharp$ away from the four vertices. In particular, \eqref{02.15.1} does not assert that all unweighted second-order derivatives are square integrable near the corner points. Possible corner singularities are recorded in the weights $r^{\al/2}$ and $r^{1+\al/2}$, and the estimate is used below only in this weighted form. Integrating over $(0,T)$ and using
	\[
		\|\mcA\vp\|_{L^2(Q)}\leq \|\pt_{tt}\vp\|_{L^2(Q)}+\|f\|_{L^2(Q)},
	\]
	together with \eqref{03.12.2}, we obtain \eqref{02.15.2} and \eqref{02.15.1}.
\end{proof}

\begin{corollary}[Weighted trace and vanishing identities]\label{02.15.C2}
	Let $\vp^0\in D(\mcA)$, $\vp^1\in H_0^1(\Om;w)$, and $f\in H^1(0,T;L^2(\Om))$. Let $\vp$ be the weak solution of \eqref{01.10.1} associated with $(\vp^0,\vp^1,f)$. Then:
	
	(i) $\pt_\theta \vp=0$ on $(\Ga_2^0\cup \Ga_2^1)\ts (0,T)$, 
	
	(ii)   $r^\f{\al}{2}\vp^2\in L^2(0,T; W^{1,1}(\Om))$, and $r^{\f{\al}{2}}\vp\in L^2(0,T; H^1(\Om))$, and $r \vp^2=r\vp=0 \mbox{ on } \Ga_2^0\ts (0,T)$, 
	
	(iii) $r(\nabla \vp\cdot A\nabla\vp)\in L^2(0,T; W^{1,1}(\Om))$, and $ r(\nabla\vp\cdot A\nabla\vp)=0$ on $\Ga_2^0\ts (0,T)$. 
\end{corollary}

\begin{proof}
	Assertion (i) follows from the fact that the Dirichlet trace of $\vp$ on $\Ga_2^0\cup \Ga_2^1$ is identically zero, so its tangential derivative with respect to $\theta$ vanishes there.
	
	For (ii), by Corollary \ref{02.14.C1} and Theorem \ref{03.11.T1},
	\[
		r^{\f{\al}{2}}\pt_\theta \vp,\ r^{\f{\al}{2}}\pt_r\vp,\ r^{\f{\al}{2}-1}\vp\in L^2(Q),
	\]
	where the last term is a consequence of Lemma \ref{03.11.L1}. Therefore
	\[
		\nabla(r^{\f{\al}{2}}\vp)=\left(r^{\f{\al}{2}}\pt_\theta\vp,\ \f{\al}{2}r^{\f{\al}{2}-1}\vp+r^{\f{\al}{2}}\pt_r\vp\right)\in L^2(Q),
	\]
	which proves $r^{\f{\al}{2}}\vp\in L^2(0,T;H^1(\Om))$. Likewise,
	\[
		\nabla(r^{\f{\al}{2}}\vp^2)=\left(2r^{\f{\al}{2}}\vp\pt_\theta\vp,\ \f{\al}{2}r^{\f{\al}{2}-1}\vp^2+2r^{\f{\al}{2}}\vp\pt_r\vp\right)\in L^2(0,T;L^1(\Om)),
	\]
	so $r^{\f{\al}{2}}\vp^2\in L^2(0,T;W^{1,1}(\Om))$. Since $\vp=0$ on $\pt\Om$, we also have $r\vp=r\vp^2=0$ on $\Ga_2^0\ts (0,T)$.
	
	For (iii), set
	\[
		F=r(\nabla\vp\cdot A\nabla\vp)=r(\pt_\theta\vp)^2+r^{1+\al}(\pt_r\vp)^2.
	\]
	By Corollary \ref{02.14.C1},
	\[
		\pt_\theta F
		=2r(\pt_\theta\vp)(\pt_{\theta\theta}\vp)+2r^{1+\al}(\pt_r\vp)(\pt_{\theta r}\vp)\in L^2(0,T;L^1(\Om)),
	\]
	and
	\[
		\pt_r F
		=(\pt_\theta\vp)^2+2r(\pt_\theta\vp)(\pt_{\theta r}\vp)+(1+\al)r^\al(\pt_r\vp)^2+2r^{1+\al}(\pt_r\vp)(\pt_{rr}\vp)\in L^2(0,T;L^1(\Om)).
	\]
	Hence $F\in L^2(0,T;W^{1,1}(\Om))$. Finally, the factor $r$ gives $F=0$ on $\Ga_2^0\ts (0,T)$. This completes the proof.
\end{proof}

\begin{remark}\label{02.15.R3}
	Corollaries \ref{02.14.C1} and \ref{02.15.C2} are global weighted regularity statements. They do not assert ordinary $H^2$ regularity up to the four corner points of $\Om$. Accordingly, the classical $H^2$ improvement is recorded only on regions of the form $\Om_\de^\sharp$, away from the four vertices, as in Theorem \ref{03.11.T1}. The weighted corollaries themselves are kept in global form because Section \ref{S3} uses the corresponding weighted trace and vanishing identities at the degenerate side $r=0$. This is also why the observation region in Section \ref{S3} is restricted to the upper side away from the two top vertices.
\end{remark}

\section{Carleman Estimate and Observability with Interior Remainder}\label{S3}

Let $\de_0\in (0,\f{1}{32})$ be fixed. Set
\[
\om=[(0,4\de_0)\cup (1-4\de_0,1)]\ts (0,1),\qquad \Om_0=(\de_0,1-\de_0)\ts (0,1),\qquad Q_0=\Om_0\ts (0,T),
\]
and denote
\[
\Ga_{2,\de_0}^1=(\de_0,1-\de_0)\ts \{1\}\subset \Ga_2^1.
\]
Let $\vp^0\in D(\mcA)$, $\vp^1\in H_0^1(\Om;w)$, and $f\in H^1(0,T; L^2(\Om))$.

Choose
\begin{equation}\label{03.26.1}
	\zeta=\zeta(\theta)\in C_0^\iy(\de_0,1-\de_0),\quad  0\leq \zeta\leq 1, 
\end{equation}
such that
\begin{equation*}
	\zeta=1 \mbox{ on } (3\de_0, 1-3\de_0),\quad \zeta=0 \mbox{ on } [0,2\de_0)\cup (1-2\de_0,1), 
\end{equation*}
and 
\begin{equation*}
	|\pt_\theta\zeta|\leq C\de_0^{-1}, \quad |\pt_{\theta\theta} \zeta|\leq C\de_0^{-2}, 
\end{equation*}
where the positive constants $C$ are absolute. 

Denote $\psi=\zeta\vp$. Then $\psi$ solves
\begin{equation}\label{03.26.2}
	\begin{cases}
		\pt_{tt}\psi-\Div(A\nabla\psi)=\zeta f-2A\nabla\zeta\cdot \nabla \vp-\vp\Div(A\nabla\zeta)\equiv h, &\mbox{in }Q_0,\\
			\psi=0, &\mbox{on }\pt\Om_0\ts(0,T),\\
		\psi(0)=\zeta\vp^0,\pt_t\psi(0)=\zeta \vp^1, &\mbox{in }\Om_0. 
	\end{cases}
\end{equation}
Since $\zeta$ is compactly supported in $(\de_0,1-\de_0)$, the traces of $\psi$ and of the conjugated variable $\eta=e^{s\si}\psi$ vanish near the two upper corner points. Consequently, whenever only $\eta$ appears, the boundary integrals on the top side may be written over the whole segment $\Ga_2^1$ without changing their value.
This is also why the final observation is formulated on $\Ga_{2,\de_0}^1$ rather than on the whole top side: near the upper corner points one cannot expect the same $H^2$-type boundary regularity as on flat boundary portions away from the vertices.

Let $s,\la$ be positive constants to be specified later, and set
\begin{equation}\label{03.26.4}
	\xi=\theta^2+r^{2-\al}-\be(t-t_0)^2, \quad \si=e^{\la\xi}, 
\end{equation}
where $t_0\in (0,T)$, and   $\be\in (0,\f{1}{2}\min\{\f{1}{8}(2-\al)^2,\de_0\})$ are positive constants.  Denote
\begin{equation}\label{04.02.1}
	b(\xi)(z,t)=|\xi_t(z,t)|^2-A\nabla\xi\cdot \nabla\xi=4\be^2(t-t_0)^2-\left[4\theta^2+(2-\al)^2r^{2-\al}\right]. 
\end{equation}

\begin{equation}\label{03.26.5}
	\eta=e^{s\si}\psi\ \Lra \ \psi=\eta e^{-s\si}. 
\end{equation}
Under this conjugation, equation \eqref{03.26.2} takes the form
\begin{equation}\label{03.26.6}
	\begin{split}
		e^{s\si}h=P^+\eta+P^-\eta, 
	\end{split}
\end{equation}
where 
\begin{equation}\label{03.26.7}
	\begin{split}
		P^+\eta
		&\equiv \sum_{i=1}^2P_i^+\eta=\left[\eta_{tt}-\Div(A\nabla\eta)\right]+s^2\eta\left[(\si_t)^2-A\nabla\si\cdot\nabla\si\right],\\
		P^-\eta
		&\equiv \sum_{i=1}^2P_i^-\eta=2s\left[-\eta_t\si_t+\nabla\eta \cdot A\nabla\si\right]+s\eta \left[-\si_{tt}+\Div(A\nabla\si)\right]. 
	\end{split}
\end{equation}

By Theorem \ref{03.11.T1}, we have
\begin{equation}\label{03.26.8}
	\vp\in L^2(0,T; D(\mcA))\cap H^1(0,T; H_0^1(\Om;w))\cap H^2(0,T; L^2(\Om)), 
\end{equation}
and then 
\begin{equation}\label{03.26.9}
	\eta \in L^2(0,T; D(\mcA))\cap H^1(0,T; H_0^1(\Om;w))\cap H^2(0,T; L^2(\Om)).
\end{equation}
Corollary \ref{02.14.C1} yields
\begin{equation}\label{03.26.12}
	\begin{split} 
		\pt_{\theta\theta}\eta\in L^2(Q),\ r^\f{\al}{2}\pt_{\theta r}\eta \in L^2(Q),\ r^{1+\f{\al}{2}}\pt_{rr}\eta\in L^2(Q),\mbox{ and } \pt_r(r^\al\pt_r\eta)\in L^2(Q), 
	\end{split} 
\end{equation}
and then 
\begin{equation}\label{03.26.11}
	\pt_\theta\eta\in H_0^1(\Om;w), \mbox{ and } r^\al \pt_r\eta\in H^1(\Om). 
\end{equation}
Moreover, by the definition of $\eta$, we have
\begin{equation}\label{03.26.10}
	\eta=0 \mbox{ on } [(0,2\de_0)\cup (1-2\de_0,1)]\ts (0,1)\ts (0,T). 
\end{equation} 

All subsequent identities are first justified for smooth approximating solutions, so that every integration by parts is legitimate at the approximation level. The regularity recorded in \eqref{03.26.8}--\eqref{03.26.12}, together with the uniform estimates established in Section \ref{S2}, allows one to pass to the present solution class by density.

The choice of $\xi$ is adapted to the anisotropic principal part. Indeed,
\[
	A\nabla\xi=(2\theta,(2-\al)r),\qquad A\nabla\xi\cdot\nabla\xi=4\theta^2+(2-\al)^2r^{2-\al},
\]
so the radial component $r^{2-\al}$ produces a nondegenerate first-order factor despite the coefficient $r^\al$ in the operator.
Moreover,
\[
	D^2\xi=
	\begin{pmatrix}
		2 & 0\\
		0 & (2-\al)(1-\al)r^{-\al}
	\end{pmatrix},
\]
so the radial Hessian contributes a positive factor $(2-\al)(1-\al)r^{-\al}$ throughout the subcritical range $\al\in (0,1)$. This is the structural reason why the present Carleman weight is compatible with the degenerate operator and why the endpoint $\al=1$ lies outside the present framework.
The tangential term $\theta^2$ is equally essential: it is precisely this component that generates the positive angular contribution in the bulk estimates below, so the weight is not purely radial.

\subsection{Preliminary Computations}

From \eqref{03.26.4}, we obtain
\begin{equation}\label{03.26.13}
	\nabla\xi=\left(2\theta,(2-\al)r^{1-\al}\right), \quad \pt_t\xi=-2\be (t-t_0),
\end{equation}
and
\begin{equation}\label{03.29.1}
	\begin{split} 
		A\nabla\xi
		&=\left(2\theta, (2-\al)r\right). 
	\end{split} 
\end{equation}

From \eqref{03.26.4} and \eqref{03.26.13}, we obtain
\begin{equation}\label{03.31.2}
	\nabla\si=\la \si \nabla\xi=\la\si \left(2\theta,(2-\al)r^{1-\al}\right),\quad \pt_t\si=\la\si \xi_t=-2\be\la\si(t-t_0), 
\end{equation}
and 
\begin{equation}\label{03.31.3}
	\begin{split}
		A\nabla\si
		&=\la \si A\nabla\xi=\la \si(2\theta, (2-\al)r),\\  \nabla\eta\cdot A\nabla\si
		&=\la\si\left[2\theta\pt_\theta \eta+(2-\al)r\pt_r\eta\right],\\
		A\nabla\si\cdot \nabla\si
		&=\la^2\si^2\left[4\theta^2+(2-\al)^2r^{2-\al}\right], 
	\end{split}
\end{equation}
and 
\begin{equation}\label{04.01.3}
	\begin{split}
		\Div(A\nabla\si)
		&= (4-\al)\la\si +\la^2\si\left[4\theta^2+(2-\al)^2r^{2-\al}\right], \\
		\nabla \Div(A\nabla\si)
		&=\la^2\si\left((16-2\al)\theta, (2-\al)(8-3\al+\al^2)r^{1-\al}\right)\\
		&\hspace{4.5mm}+\la^3\si\left[4\theta^2+(2-\al)^2r^{2-\al}\right]\left(2\theta, (2-\al)r^{1-\al}\right), 
	\end{split}
\end{equation}
and 
\begin{equation}\label{03.31.4}
	\begin{split}
		D^2\si
		&=\la\si 
		\begin{pmatrix}
			2 &\\
			& (2-\al)(1-\al)r^{-\al}
		\end{pmatrix}\\
		&\hspace{4.5mm}+\la^2\si 
		\begin{pmatrix}
			4\theta^2 &2(2-\al)\theta r^{1-\al}\\
			2(2-\al)\theta r^{1-\al} & (2-\al)^2r^{2-2\al}
		\end{pmatrix},
	\end{split}
\end{equation}
and 
\begin{equation}\label{04.01.1}
	\begin{split}
		\si_{tt}=\la\si\xi_{tt}+\la^2\si(\xi_t)^2=-2\be \la\si+4\be^2\la^2\si(t-t_0)^2, 
	\end{split}
\end{equation}
and 
\begin{equation}\label{04.01.2}
	\begin{split}
		\nabla\si_{tt}
		&=\left[\la\xi_{tt}+\la^2(\xi_t)^2\right]\nabla\si =\left[-2\be\la^2+4\be^2\la^3(t-t_0)^2\right]\si\left(2\theta, (2-\al)r^{1-\al}\right),
	\end{split}
\end{equation} 
and 
\begin{equation}\label{04.02.3}
	\begin{split}
		(\si_t)^2-A\nabla\si\cdot \nabla\si
		&=\la^2\si^2(\xi_t)^2-\la^2\si^2 A\nabla\xi\cdot \nabla\xi=\la^2\si^2b(\xi)
	\end{split}
\end{equation}
and
\begin{equation}\label{04.02.2}
	\begin{split}
		\si_{tt}-\Div(A\nabla\si)
		&=-(4-\al+2\be)\la\si+\la^2\si b(\xi)
	\end{split}
\end{equation} 

Since $\eta\in L^2(0,T; H_0^1(\Om_0;w))$, we get: 

(I) $\eta^2\in L^1(Q_0)$, and $\pt_\theta \eta^2=2\eta\pt_\theta \eta\in L^1(Q_0)$, which shows that $\eta^2\in C_0(\de_0, 1-\de_0)$ for a.e.~$(r,t)\in (0,1)\ts (0,T)$ by \eqref{03.26.11}, and 
\begin{equation}\label{03.19.1}
	\begin{split}
		\iint_{Q_0} \pt_\theta \eta^2\df z\df t=\int_0^T\int_0^1\int_{\de_0}^{1-\de_0}\pt_\theta \eta^2\df \theta \df r\df t=0
	\end{split}
\end{equation}
by Fubini's Theorem and \eqref{03.26.10}; 

(II) $\eta^2\in L^1(Q_0)$, and $\pt_r (r\eta^2)=\eta^2+2r\eta \pt_r\eta\in L^1(Q_0)$, which shows that $r\eta^2\in C_0[0,1]$ for a.e.~$(\theta,t)\in (\de_0,1-\de_0)\ts (0,T)$ by \eqref{03.26.11}, and 
\begin{equation}\label{03.19.2}
	\begin{split}
		\iint_{Q_0} \pt_r(r\eta^2)\df z\df t=\int_0^T\int_{\de_0}^{1-\de_0}\int_0^1 \pt_r(r\eta^2)\df r\df \theta \df t=0
	\end{split}
\end{equation}
by Fubini's Theorem and Corollary \ref{02.15.C2} (ii). 

Since $\pt_\theta \eta\in H_0^1(\Om;w)$, we get: 

(III) $(\pt_\theta\eta)^2\in L^1(Q_0)$, and $\pt_\theta (\pt_\theta\eta)^2=2\pt_{\theta}\eta \pt_{\theta\theta}\eta\in L^1(Q_0)$, which shows that $(\pt_\theta\eta)^2\in C_0(\de_0,1-\de_0)$ for a.e.~$(r,t)\in (0,1)\ts (0,T)$ by \eqref{03.26.12}, and 
\begin{equation}\label{03.19.3}
	\iint_{Q_0} \pt_\theta (\pt_\theta\eta)^2\df z\df t=\int_0^T\int_0^1\int_{\de_0}^{1-\de_0}\pt_\theta (\pt_\theta \eta)^2\df \theta \df r\df t=0
\end{equation}
by Fubini's Theorem and \eqref{03.26.10}; 

(IV) $(\pt_\theta\eta)^2\in L^1(Q_0)$, and $\pt_r[r(\pt_\theta \eta)^2]=(\pt_\theta\eta)^2+2r(\pt_\theta \eta)(\pt_{r\theta}\eta)\in L^1(Q_0)$ by \eqref{03.26.12}, which shows that $r(\pt_\theta \eta)^2\in C_0[0,1]$ for a.e.~$(\theta,t)\in (\de_0,1-\de_0)\ts (0,T)$ by Corollary \ref{02.15.C2} (i) and  (iii), and 
\begin{equation}\label{03.19.4}
	\begin{split}
		\iint_{Q_0} \pt_r \left[r(\pt_\theta \eta)^2\right]\df z\df t=\int_0^T\int_{\de_0}^{1-\de_0}\int_0^1\pt_r \left[r(\pt_\theta\eta)^2\right]\df \theta \df r\df t=0
	\end{split}
\end{equation}
by Fubini's Theorem and Corollary \ref{02.15.C2} (i) and (iii). 

	(V) Since $r^\al (\pt_r\eta)^2\in L^1(Q_0)$ and $\pt_\theta [r^\al (\pt_r\eta)^2]=2r^\al (\pt_r\eta)(\pt_{\theta r}\eta)\in L^1(Q_0)$ by \eqref{03.26.12}, which shows that 
$r^\al (\pt_r\eta)^2\in C_0(\de_0,1-\de_0)$ for a.e.~$(r,t)\in (0,1)\ts (0,T)$, and 
\begin{equation}\label{03.19.5}
	\begin{split}
		\iint_{Q_0}\pt_\theta \left[r^\al (\pt_r\eta)^2\right]\df z\df t=\int_0^T\int_0^1\int_{\de_0}^{1-\de_0}\pt_\theta \left[r^\al (\pt_r\eta)^2\right]\df \theta \df r\df t=0
	\end{split}
\end{equation}
by Fubini's Theorem and \eqref{03.26.10}. 

(VI) Since $r^{\al+1}(\pt_r\eta)^2\in L^1(Q_0)$ and $\pt_r[r^{\al+1}(\pt_r\eta)^2]=(\al+1)r^\al(\pt_r\eta)^2+2r^{\al+1}(\pt_r\eta)(\pt_{rr}\eta)\in L^1(Q_0)$ by \eqref{03.26.12}, which shows that $r^{\al+1}(\pt_r\eta)^2\in C[\de_0,1-\de_0]$ for a.e.~$(\theta,t)\in (\de_0,1-\de_0)\ts (0,T)$, and 
\begin{equation}\label{03.19.6}
	\begin{split}
		\iint_{Q_0}\pt_r\left[r^{\al+1}(\pt_r\eta)^2\right]\df z\df t
		&=\int_0^T\int_{\de_0}^{1-\de_0}\int_0^1\pt_r\left[r^{\al+1}(\pt_r\eta)^2\right]\df r\df \theta \df t\\
		&=\iint_{\Ga_2^1\ts (0,T)} (\pt_r\eta)^2\df S\df t
	\end{split}
\end{equation}
by Fubini's Theorem and Corollary \ref{02.15.C2} (iii).

\subsection{Cross-Term Computations}

In this subsection we compute $(P_i^+ \eta, P_j^- \eta)_{L^2(Q_0)}$, for $(i,j)\in\{1,2\}^2$, term by term. Since integration by parts plays a crucial role, we indicate explicitly where it is used. We also write $x_1=\theta$ and $x_2=r$. 
No artificial radial boundary is introduced in Section \ref{S3}: the only radial boundary pieces come from $\Ga_2^0$ and $\Ga_2^1$. The terms on $\Ga_2^0$ vanish because of the factors $r$ or $r^{1+\al}$ together with Corollary \ref{02.15.C2}, so the only surviving observation term is carried by $\Ga_2^1$.

\medskip\noindent{\it Step 0: Auxiliary terminal-time condition.} We further assume
\begin{equation}\label{03.02.4}
	\eta(t)=\pt_t\eta(t)=0 \mbox{ for } t=0,T. 
\end{equation} 
This assumption will be removed later. In particular,
\begin{equation*}
	\nabla\eta(t)=\pt_t\nabla\eta(t)=0 \mbox{ for } t=0,T. 
\end{equation*}

\medskip\noindent{\it Step 1: Computation of $(P_1^+\eta, P_1^-\eta)_{L^2(Q_0)}$.}

	A direct expansion gives
\begin{equation*}
	\begin{split}
		(P_1^+\eta, P_1^-\eta)_{L^2(Q_0)}
		&=-2s\iint_{Q_0}\eta_{tt}\eta_t\si_t\df z\df t+2s\iint_{Q_0}\eta_{tt}\nabla \eta\cdot A\nabla\si\df z\df t\\
		&\hspace{4.5mm}+2s\iint_{Q_0}\left[\Div(A\nabla\eta)\right]\eta_t\si_t\df z\df t-2s\iint_{Q_0}\left[\Div(A\nabla\eta)\right](\nabla\eta \cdot A\nabla\si)\df z\df t\\
		&\equiv \sum_{i=1}^4P_{11,i}. 
	\end{split}
\end{equation*}

	Using \eqref{03.02.4}, we obtain
\begin{equation*}
	\begin{split}
		P_{11,1}
		&=s\iint_{Q_0}(\eta_t)^2\si_{tt}\df z\df t. 
	\end{split}
\end{equation*}

	By $\eta_t\in L^2(0,T; H_0^1(\Om;w))$, \eqref{03.19.1}, and \eqref{03.19.2},
\begin{equation*}
	\iint_{Q_0}\Div\left[(\eta_t)^2A\nabla\si\right]\df S\df t=0. 
\end{equation*} 
	Using again \eqref{03.02.4}, we obtain
\begin{equation*}
	\begin{split}
		P_{11,2}
		&=-2s\iint_{Q_0} \eta_t\nabla \eta_t\cdot A\nabla\si\df z\df t-2s\iint_{Q_0} \eta_t\nabla\eta \cdot A\nabla\si_t\df z\df t\\
		&=s\iint_{Q_0}(\eta_t)^2\Div(A\nabla\si)\df z\df t-2s\iint_{Q_0}\eta_t\nabla\eta \cdot A\nabla\si_t\df z\df t. 
	\end{split}
\end{equation*}

	From $\eta_t\in L^2(0,T; H_0^1(\Om;w))$ and \eqref{03.02.4}, we obtain
\begin{equation*}
	\begin{split}
		P_{11,3}
		&=-2s\iint_{Q_0}\si_tA\nabla\eta \cdot \nabla\eta_t\df z\df t-2s\iint_{Q_0} \eta_tA\nabla\eta \cdot \nabla\si_t\df z\df t\\
		&=s\iint_{Q_0}(A\nabla\eta\cdot \nabla\eta)\si_{tt}\df z\df t-2s\iint_{Q_0}\eta_t\nabla\eta \cdot A\nabla \si_t\df z\df t. 
	\end{split}
\end{equation*}

From \eqref{03.26.11}, \eqref{03.31.3}, and $\al\in (0,1)$, we infer that $\nabla\eta \cdot A\nabla\si\in H^1(\Om;w)$. Moreover, by Corollary \ref{02.15.C2} (iii),
\begin{equation*}
		\iint_{\Ga_2^0\ts (0,T)}(r\pt_r\eta)^2\df S\df t\leq \iint_{\Ga_2^0\ts (0,T)}r^{1+\al}(\pt_r\eta)^2\df S\df t=0, 
\end{equation*}
and then $r\pt_r\eta=0$ on $\Ga_2^0\ts (0,T)$. Together with these and Corollary \ref{02.15.C2} (i), 
	we get $\nabla\eta\cdot A\nabla\si=0$ on $\Ga_2^0\ts (0,T)$. Then
\begin{equation*}
		\begin{split}
			P_{11,4}
				&=-2s\iint_{\Ga_2^1\ts (0,T)}(A\nabla\eta\cdot \nu)(\nabla\eta\cdot A\nabla\si)\df S\df t \\
			&\hspace{6mm}+2s\iint_{Q_0}A\nabla\eta \cdot \nabla(\nabla\eta \cdot A\nabla\si)\df z\df t\\
			&=-2s\iint_{\Ga_2^1\ts (0,T)}(A\nabla\eta\cdot \nu)(\nabla\eta\cdot A\nabla\si)\df S\df t\\
			&\hspace{6mm}+2s\sum_{i,j=1}^2\iint_{Q_0}A_i(\pt_{x_i}\eta)(\pt_{x_j}\eta)(\pt_{x_i}A_j)(\pt_{x_j}\si)\df z\df t\\
			&\hspace{6mm}+2s\iint_{Q_0}A\nabla\eta\cdot \left[D^2\si(A\nabla\eta)\right]\df z\df t\\
			&\hspace{6mm}+2s\sum_{i,j=1}^2\iint_{Q_0}A_i(\pt_{x_i}\eta)(\pt_{x_ix_j}\eta)A_j(\pt_{x_j}\si)\df z\df t.
		\end{split}
	\end{equation*}
Note that Corollary \ref{02.15.C2} (i) and \eqref{03.31.3} yield
\begin{equation*}
	\begin{split}
			-2s\iint_{\Ga_2^1\ts (0,T)}(A\nabla\eta\cdot \nu)(\nabla\eta\cdot A\nabla\si)\df S\df t=-2(2-\al)s\la\iint_{\Ga_2^1\ts (0,T)} \si (\pt_r\eta)^2\df S\df t, 
	\end{split}
\end{equation*}
and 
\begin{equation*}
	\begin{split}
		&2s\sum_{i,j=1}^2\iint_{Q_0}A_i(\pt_{x_i}\eta)(\pt_{x_j}\eta)(\pt_{x_i}A_j)(\pt_{x_j}\si)\df z\df t=2(2-\al)\al s\la\iint_{Q_0}\si r^{\al}(\pt_r\eta)^2\df z\df t,
	\end{split}
\end{equation*}
and 
\begin{equation*}
	\begin{split}
		&2s\iint_{Q_0}A\nabla\eta\cdot \left[D^2\si(A\nabla\eta)\right]\df z\df t\\
		&=2s\la \iint_{Q_0} \si  \left[2(\pt_\theta \eta)^2+(2-\al)(1-\al)r^\al (\pt_r\eta)^2\right]\df z\df t\\
		&\hspace{4.5mm}+2s\la^2\iint_{Q_0}\si\left[2\theta\pt_\theta \eta+(2-\al)r\pt_r\eta\right]^2\df z\df t
	\end{split}
\end{equation*}
by \eqref{03.31.4}. 
We compute
\begin{equation*}
	\begin{split}
		&2s\sum_{i,j=1}^2\iint_{Q_0}A_i(\pt_{x_i}\eta)(\pt_{x_ix_j}\eta)A_j(\pt_{x_j}\si)\df z\df t\\
		&=s\sum_{i,j=1}^2\iint_{Q_0} A_i\left[\pt_{x_j}(\pt_{x_i}\eta)^2\right] A_j(\pt_{x_j}\si)\df z\df t\\
		&=s\iint_{Q_0} (\pt_\theta \si)\pt_\theta (\pt_\theta \eta)^2\df z\df t+s\iint_{Q_0} r^\al (\pt_r\si)\pt_r(\pt_\theta \eta)^2\df z\df t\\
		&\hspace{4.5mm}+s\iint_{Q_0} r^\al  (\pt_\theta \si)\pt_\theta(\pt_r\eta)^2\df z\df t+s\iint_{Q_0} r^{2\al}(\pt_r\si)\pt_r(\pt_r\eta)^2\df z\df t, 
	\end{split}
\end{equation*}
	Using \eqref{03.31.2}, \eqref{03.19.3}, and \eqref{03.19.4}, we obtain
\begin{equation*}
	\begin{split}
		s\iint_{Q_0}(\pt_\theta \si)\pt_\theta(\pt_\theta\eta)^2\df z\df t
		&=2s\la\iint_{Q_0}\si\theta\pt_\theta(\pt_\theta\eta)^2\df z\df t\\
		&=2s\la\iint_{Q_0}\pt_\theta \left[\si\theta(\pt_\theta\eta)^2\right]\df z\df t\\
		&\hspace{4.5mm}-2s\la\iint_{Q_0} \si(\pt_\theta \eta)^2\df z\df t-4s\la^2\iint_{Q_0}\si \theta^2(\pt_\theta\eta)^2\df z\df t\\
		&=-2s\la\iint_{Q_0} \si(\pt_\theta \eta)^2\df z\df t-4s\la^2\iint_{Q_0}\si \theta^2(\pt_\theta\eta)^2\df z\df t, 
	\end{split}
\end{equation*}
	Using \eqref{03.31.2} and \eqref{03.19.4}, we obtain
\begin{equation*}
	\begin{split}
		&s\iint_{Q_0}r^\al (\pt_r\si)\pt_r(\pt_\theta\eta)^2\df z\df t\\
		&=(2-\al)s\la\iint_{Q_0} \si r\pt_r(\pt_\theta\eta)^2\df z\df t\\
		&=(2-\al)s\la\iint_{Q_0}\pt_r\left[\si r(\pt_\theta \eta)^2\right]\df z\df t\\
		&\hspace{4.5mm}-(2-\al)s\la\iint_{Q_0}\si (\pt_\theta \eta)^2\df z\df t-(2-\al)^2s\la^2\iint_{Q_0} \si r^{2-\al}(\pt_\theta\eta)^2\df z\df t\\
		&=-(2-\al)s\la\iint_{Q_0}\si (\pt_\theta \eta)^2\df z\df t-(2-\al)^2s\la^2\iint_{Q_0} \si r^{2-\al}(\pt_\theta\eta)^2\df z\df t, 
	\end{split}
\end{equation*}
	Using \eqref{03.31.2} and \eqref{03.19.5}, we obtain
\begin{equation*}
	\begin{split}
		&s\iint_{Q_0}r^\al (\pt_\theta\si)\pt_\theta (\pt_r\eta)^2\df z\df t\\
		&=2s\la\iint_{Q_0}\si \theta r^\al \pt_\theta (\pt_r\eta)^2\df z\df t\\
		&=2s\la\iint_{Q_0}\pt_\theta \left[\si\theta r^\al (\pt_r\eta)^2\right]\df z\df t\\
		&\hspace{4.5mm}-2s\la\iint_{Q_0}\si r^\al (\pt_r\eta)^2\df z\df t-4s\la^2\iint_{Q_0} \si \theta^2r^\al (\pt_r\eta)^2\df z\df t\\
		&=-2s\la\iint_{Q_0}\si r^\al (\pt_r\eta)^2\df z\df t-4s\la^2\iint_{Q_0} \si \theta^2r^\al (\pt_r\eta)^2\df z\df t, 
	\end{split}
\end{equation*}
	Using \eqref{03.31.2} and \eqref{03.19.6}, we obtain
	\begin{equation*}
		\begin{split}
			&s\iint_{Q_0} r^{2\al}(\pt_r\si)\pt_r(\pt_r\eta)^2\df z\df t\\
			&=(2-\al)s\la\iint_{Q_0}\si r^{1+\al}\pt_r(\pt_r\eta)^2\df z\df t\\
		&=(2-\al)s\la\iint_{Q_0}\pt_r\left[\si r^{1+\al}(\pt_r\eta)^2\right]\df z\df t\\
		&\hspace{4.5mm}-(2-\al)(1+\al)s\la\iint_{Q_0} \si r^\al (\pt_r\eta)^2\df z\df t-(2-\al)^2s\la^2\iint_{Q_0}\si r^{2-\al}r^\al (\pt_r\eta)^2\df z\df t\\
			&=(2-\al)s\la\iint_{\Ga_2^1\ts (0,T)} \si (\pt_r\eta)^2\df S\df t\\
			&\hspace{4.5mm}-(2-\al)(1+\al)s\la\iint_{Q_0} \si r^\al (\pt_r\eta)^2\df z\df t-(2-\al)^2s\la^2\iint_{Q_0}\si r^{2-\al}r^\al (\pt_r\eta)^2\df z\df t,
		\end{split}
	\end{equation*}
	where the contribution at $r=0$ vanishes because $r^{1+\al}(\pt_r\eta)^2$ is controlled by $r(\nabla\eta\cdot A\nabla\eta)$ and Corollary \ref{02.15.C2}(iii) yields the corresponding vanishing trace on $\Ga_2^0\ts (0,T)$.
		Combining the preceding identities gives
	\begin{equation*}
		\begin{split}
		P_{11,4}
			&=-(2-\al)s\la\iint_{\Ga_2^1\ts (0,T)}\si (\pt_r\eta)^2\df S\df t\\
		&\hspace{4.5mm}+s\la\iint_{Q_0}\si\left[\al(\pt_\theta \eta)^2+(-3\al+\al^2)r^\al(\pt_r\eta)^2\right]\df z\df t\\
		&\hspace{4.5mm}+2s\la^2\iint_{Q_0}\si\left[2\theta\pt_\theta \eta+(2-\al)r\pt_r\eta\right]^2\df z\df t\\
		&\hspace{4.5mm}-s\la^2\iint_{Q_0}\si \left[4\theta^2+(2-\al)^2r^{2-\al} \right]\left[(\pt_\theta \eta)^2+r^\al (\pt_r\eta)^2\right]\df z\df t. 
	\end{split}
\end{equation*}

\medskip\noindent{\it Step 2: Computation of $(P_1^+\eta, P_2^-\eta)_{L^2(Q_0)}$.}

A direct expansion gives
\begin{equation*}
	\begin{split}
		(P_1^+\eta, P_2^-\eta)_{L^2(Q_0)}
		&=-s\iint_{Q_0} \eta_{tt}\eta \si_{tt}\df z\df t+s\iint_{Q_0} \eta_{tt}\eta \Div(A\nabla\si)\df z\df t\\
		&\hspace{4.5mm}+s\iint_{Q_0} \left[\Div(A\nabla\eta)\right]\eta \si_{tt}\df z\df t-s\iint_{Q_0} \left[\Div(A\nabla\eta)\right]\eta \Div(A\nabla\si)\df z\df t\\
		&\equiv \sum_{i=1}^4P_{12,i}. 
	\end{split}
\end{equation*}

Using \eqref{03.02.4}, we obtain
\begin{equation*}
	\begin{split}
		P_{12,1}
		&=s\iint_{Q_0} (\eta_t)^2\si_{tt}\df z\df t-\f{s}{2}\iint_{Q_0} \eta^2\si_{tttt}\df z\df t. 
	\end{split}
\end{equation*}

Using \eqref{03.02.4}, we obtain
\begin{equation*}
	\begin{split}
		P_{12,2}
		&=-s\iint_{Q_0}(\eta_t)^2\Div(A\nabla\si)\df z\df t+\f{s}{2}\iint_{Q_0}\eta^2\pt_{tt}\Div(A\nabla\si)\df z\df t. 
	\end{split}
\end{equation*}

Using $\eta\in L^2(0,T; H_0^1(\Om;w))$, \eqref{03.19.1}, \eqref{03.19.2}, and \eqref{04.01.2}, we obtain
\begin{equation*}
	\begin{split}
		P_{12,3}
		&=-s\iint_{Q_0} (A\nabla\eta\cdot \nabla\eta)\si_{tt}\df z\df t-s\iint_{Q_0}\eta \nabla\eta\cdot A\nabla\si_{tt}\df z\df t\\
		&=-s\iint_{Q_0}(A\nabla\eta\cdot \nabla\eta)\si_{tt}\df z\df t+\f{s}{2}\iint_{Q_0}\eta^2\pt_{tt}\Div(A\nabla\si)\df z\df t. 
	\end{split}
\end{equation*}

Using $\eta\in L^2(0,T; H_0^1(\Om;w))$, \eqref{04.01.3}, \eqref{03.19.1}, and \eqref{03.19.2}, we obtain
\begin{equation*}
	\begin{split}
		P_{12,4}
		&=s\iint_{Q_0} A\nabla\eta \cdot \nabla\eta \Div(A\nabla\si)\df z\df t+s\iint_{Q_0}\eta \nabla\eta \cdot A\nabla \Div(A\nabla\si)\df z\df t\\
		&=s\iint_{Q_0}A\nabla\eta\cdot \nabla \eta\left[\Div(A\nabla\si)\right]\df z\df t-\f{s}{2}\iint_{Q_0}\eta^2\mcA^2\si\df z\df t. 
	\end{split}
\end{equation*}

\medskip\noindent{\it Step 3: Computation of $(P_2^+\eta,P_1^-\eta)_{L^2(Q_0)}$.}

A direct expansion gives
	\begin{equation*}
		\begin{split}
			(P_2^+\eta, P_1^-\eta)_{L^2(Q_0)}
			&=-2s^3\iint_{Q_0}\eta\eta_t(\si_t)^3\df z\df t
			+2s^3\iint_{Q_0} \eta (\nabla\eta \cdot A\nabla\si )(\si_t)^2\df z\df t\\
			&\hspace{6mm}+2s^3\iint_{Q_0} \eta \eta_t \si_t A\nabla\si\cdot \nabla\si\df z\df t\\
			&\hspace{6mm}-2s^3\iint_{Q_0} \eta (\nabla\eta \cdot A\nabla\si)(A\nabla\si\cdot \nabla\si)\df z\df t\\
			&\equiv \sum_{i=1}^4P_{21,i}. 
		\end{split}
	\end{equation*}

Using \eqref{03.02.4}, we obtain
\begin{equation*}
	\begin{split}
		P_{21,1}
		&=3s^3\iint_{Q_0} \eta^2(\si_t)^2\si_{tt}\df z\df t. 
	\end{split}
\end{equation*}

Using \eqref{03.31.2}, \eqref{03.31.3}, \eqref{03.19.1}, and \eqref{03.19.2}, we obtain
\begin{equation*}
	\begin{split}
		P_{21,2}
		&=-s^3\iint_{Q_0} \eta^2(\si_t)^2\Div(A\nabla\si)\df z\df t-2s^3\iint_{Q_0}\eta^2\si_t\nabla\si \cdot A\nabla\si_t\df z\df t. 
	\end{split}
\end{equation*}

Using \eqref{03.02.4}, we obtain
\begin{equation*}
	\begin{split}
		P_{21,3}
		&=-s^3\iint_{Q_0} \eta^2\si_{tt}A\nabla\si\cdot \nabla\si\df z\df t-2s^3\iint_{Q_0}\eta^2\si_tA\nabla\si\cdot \nabla\si_t\df z\df t. 
	\end{split}
\end{equation*}

Using \eqref{03.31.3}, \eqref{03.19.1}, and \eqref{03.19.2}, we obtain
\begin{equation*}
	\begin{split}
		P_{21,4}
		&=s^3\iint_{Q_0}\eta^2\left[\Div(A\nabla\si)\right]A\nabla\si\cdot \nabla\si\df z\df t+s^3\iint_{Q_0} \eta^2 A\nabla\si\cdot \nabla(A\nabla\si\cdot \nabla\si)\df z\df t. 
	\end{split}
\end{equation*}

\medskip\noindent{\it Step 4: Computation of $(P_2^+\eta, P_2^-\eta)_{L^2(Q_0)}$.}

A direct expansion gives
\begin{equation*}
	\begin{split}
		(P_2^+\eta,P_2^-\eta)_{L^2(Q_0)}
		&=-s^3\iint_{Q_0}\eta^2(\si_t)^2\si_{tt}\df z\df t+s^3\iint_{Q_0} \eta^2(\si_t)^2\Div(A\nabla\si)\df z\df t\\
		&\hspace{4.5mm}+s^3\iint_{Q_0} \eta^2\si_{tt}A\nabla\si\cdot \nabla\si\df z\df t-s^3\iint_{Q_0}\eta^2\left[\Div(A\nabla\si)\right]A\nabla\si\cdot \nabla\si\df z\df t. 
	\end{split}
\end{equation*}

\medskip\noindent{\it Step 5: Combination of Steps 1--4.} We obtain
\begin{equation*}
	\begin{split}
		(P^+\eta,P^-\eta)_{L^2(Q_0)}
		&=-(2-\al)s\la\iint_{\Ga_2^1\ts (0,T)}\si (\pt_r\eta)^2\df S\df t\\
		&\hspace{4.5mm}+2s\la^2\iint_{Q_0}\si \left[2\theta\pt_\theta \eta+(2-\al)r\pt_r\eta\right]^2\df z\df t\\
		&\hspace{4.5mm}+2s\iint_{Q_0}(\eta_t)^2\si_{tt}\df z\df t-4s\iint_{Q_0} \eta_t\nabla\eta \cdot A\nabla\si_t\df z\df t\\
		&\hspace{4.5mm}+s\la\iint_{Q_0}\si\left[4(\pt_\theta \eta)^2+(2-\al)^2r^\al(\pt_r\eta)^2\right]\df z\df t\\
		&\hspace{4.5mm}-\f{s}{2}\iint_{Q_0} \eta^2(\pt_{tt}+\mcA)^2\si\df z\df t\\
		&\hspace{4.5mm}+2s^3\iint_{Q_0}\eta^2(\si_t)^2\si_{tt}\df z\df t-4s^3\iint_{Q_0}\eta^2\si_tA\nabla\si \cdot \nabla\si_t\df z\df t\\
		&\hspace{4.5mm}+s^3\iint_{Q_0}\eta^2 A\nabla\si\cdot \nabla (A\nabla\si\cdot \nabla\si)\df z\df t.
	\end{split}
\end{equation*}

Denote
\begin{equation*}
B\equiv-(2-\al)s\la\iint_{\Ga_2^1\ts (0,T)}\si (\pt_r\eta)^2\df S\df t,
\end{equation*}
and 
\begin{equation*}
	\begin{split}
		J_1
		&\equiv2s\la^2\iint_{Q_0}\si \left[2\theta\pt_\theta \eta+(2-\al)r\pt_r\eta\right]^2\df z\df t\\
		&\hspace{4.5mm}+ 2s\iint_{Q_0}(\eta_t)^2\si_{tt}\df z\df t-4s\iint_{Q_0} \eta_t\nabla\eta \cdot A\nabla\si_t\df z\df t\\
		&\hspace{4.5mm}+s\la\iint_{Q_0}\si\left[4(\pt_\theta \eta)^2+(2-\al)^2r^\al(\pt_r\eta)^2\right]\df z\df t,
	\end{split}
\end{equation*}
and
\begin{equation*}
	\begin{split}
		J_2
		&\equiv 2s^3\iint_{Q_0}\eta^2(\si_t)^2\si_{tt}\df z\df t-4s^3\iint_{Q_0}\eta^2\si_tA\nabla\si \cdot \nabla\si_t\df z\df t\\
		&\hspace{4.5mm}+s^3\iint_{Q_0}\eta^2 A\nabla\si\cdot \nabla (A\nabla\si\cdot \nabla\si)\df z\df t,
	\end{split}
\end{equation*}
and 
\begin{equation*}
	\begin{split}
		J_3
		&\equiv -\f{s}{2}\iint_{Q_0} \eta^2(\pt_{tt}+\mcA)^2\si\df z\df t. 
	\end{split}
\end{equation*}
Then
\begin{equation*}
	\begin{split}
		(P^+\eta,P^-\eta)_{L^2(Q_0)}=B+\sum_{i=1}^3J_i. 
	\end{split}
\end{equation*}

\subsection{Completion of the Estimate}

\medskip\noindent{\it Step 6: Computation of $(P^+\eta, \si \eta)_{L^2(Q_0)}$.}

Using \eqref{03.02.4}, \eqref{03.19.1}, and \eqref{03.19.2}, we obtain
\begin{equation*}
	\begin{split}
		(P^+\eta, \si \eta)_{L^2(Q_0)}
		&=-\iint_{Q_0}\si(\eta_t)^2\df z\df t+\iint_{Q_0}\si A\nabla\eta\cdot \nabla\eta \df z\df t\\
		&\hspace{4.5mm}+\f{1}{2}\iint_{Q_0}\eta^2\si_{tt}\df z\df t-\f{1}{2}\iint_{Q_0}\eta^2\Div(A\nabla\si)\df z\df t\\
		&\hspace{4.5mm}+s^2\iint_{Q_0}\si \eta^2\left[(\si_t)^2-A\nabla\si\cdot \nabla\si\right]\df z\df t. 
	\end{split}
\end{equation*}
Using \eqref{04.02.3} and \eqref{04.02.2}, we obtain
\begin{equation*}
	\begin{split}
		(P^+\eta, \si\eta)_{L^2(Q_0)}
		&=-\iint_{Q_0}\si(\eta_t)^2\df z\df t+\iint_{Q_0}\si A\nabla\eta\cdot \nabla\eta \df z\df t\\
		&\hspace{4.5mm}-\f{4-\al+2\be}{2}\la\iint_{Q_0}\si\eta^2\df z\df t\\
		&\hspace{4.5mm}+\f{1}{2}\la^2\iint_{Q_0} \si\eta^2b(\xi)\df z\df t+s^2\la^2\iint_{Q_0} \si^3\eta^2b(\xi)\df z\df t. 
	\end{split}
\end{equation*}

\medskip\noindent{\it Step 7: Estimate of $J_1$.}

Using \eqref{03.26.13}, \eqref{03.31.2}, and \eqref{04.01.1}, we have
\begin{equation*}
	\begin{split}
		J_1
		&=2s\la\iint_{Q_0}\si (\eta_t)^2 \xi_{tt}\df z\df t+s\la\iint_{Q_0}\si\left[4(\pt_\theta\eta)^2+(2-\al)^2r^\al(\pt_r\eta)^2\right]\df z\df t\\
		&\hspace{4.5mm}+2s\la^2\iint_{Q_0} \si \left(\eta_t\xi_t-\nabla \eta\cdot A\nabla\xi\right)^2\df z\df t\\
		&\geq -4\be s\la\iint_{Q_0}\si (\eta_t)^2 \df z\df t+s\la\iint_{Q_0}\si\left[4(\pt_\theta\eta)^2+(2-\al)^2r^\al(\pt_r\eta)^2\right]\df z\df t, 
	\end{split}
\end{equation*}
Together with Step 6, we obtain
\begin{equation*}
	\begin{split}
		J_1
		&\geq \f{1}{2}\left[(2-\al)^2+4\be\right] s\la (P^+\eta, \si \eta)_{L^2(Q_0)}\\
		&\hspace{4.5mm}+\f{1}{2}\left[(2-\al)^2-4\be\right]s\la\iint_{Q_0}\si \left[(\eta_t)^2+A\nabla\eta \cdot \nabla\eta\right] \df z\df t\\
		&\hspace{4.5mm}+\f{1}{4}\left[(2-\al)^2+4\be\right] (4-\al+2\be)s\la^2\iint_{Q_0}\si\eta^2\df z\df t\\
		&\hspace{4.5mm}-\f{1}{4}\left[(2-\al)^2+4\be\right] s\la^3\iint_{Q_0} \si \eta^2b(\xi)\df z\df t\\
		&\hspace{4.5mm}-\f{1}{2}\left[(2-\al)^2+4\be\right] s^3\la^3\iint_{Q_0}\si^3\eta^2b(\xi)\df z\df t. 
	\end{split}
\end{equation*}

\medskip\noindent{\it Step 8: Estimate of $J_2$.}

Using \eqref{03.26.13}, \eqref{03.31.2}, and \eqref{03.31.3}, we obtain
\begin{equation*}
	\begin{split}
		J_2
		&\geq 2s^3\la^4\iint_{Q_0} \si^3\eta^2\left[b(\xi)\right]^2\df z\df t\\
		&\hspace{4.5mm}+s^3\la^3\iint_{Q_0}\si^3\eta^2\left[16\theta^2+(2-\al)^4r^{2-\al}-16\be^3(t-t_0)^2\right]\df z\df t.
	\end{split}
\end{equation*}

\medskip\noindent{\it Step 9: Estimate of $J_3$.}

A direct expansion of $(\pt_{tt}+\mcA)^2\si$, using \eqref{04.01.1}--\eqref{04.02.2}, shows that it only contributes lower-order terms with respect to the positive bulk terms isolated in Steps 7 and 8. More precisely, after collecting the coefficients of the explicit derivatives of $\si$, one obtains a bound of the form
\[
	|J_3|
	\leq C\left(s\la^3\iint_{Q_0}\si\eta^2\df z\df t+s^3\la^2\iint_{Q_0}\si^3\eta^2\df z\df t\right),
\]
where the positive constant $C$ depends only on $\al$, $\be$, and $T$. These terms are of lower order and will be absorbed in the next step by choosing $\la$ and $s$ sufficiently large.

\medskip\noindent{\it Step 10: Final estimate of $(P^+\eta, P^-\eta)_{L^2(Q_0)}$.}

Combining the positive contributions obtained in Steps 7 and 8 with the lower-order control from Step 9, and then applying Young's inequality to the term $(P^+\eta,\si\eta)_{L^2(Q_0)}$, we can absorb all subordinate contributions into the coercive bulk terms by taking first $\la\geq \la_0(\de_0,\al,\be,T)$ and then $s\geq 2$ sufficiently large. Using also
\[
	16\theta^2+2(2-\al)^4r^{2-\al}\geq 16\de_0^2\qquad\mbox{on }Q_0,
\]
we arrive at
\begin{equation*}
	\begin{split}
		(P^+\eta,P^-\eta)_{L^2(Q_0)}
		&\geq B+ \f{1}{2}\left[(2-\al)^2-4\be\right]s\la\iint_{Q_0}\si \left[(\eta_t)^2+A\nabla\eta \cdot \nabla\eta\right] \df z\df t\\
		&\hspace{4.5mm}+Cs^3\la^3\iint_{Q_0} \si^3\eta^2\df z\df t-\e\|P^+\eta\|_{L^2(Q_0)}^2
	\end{split}
\end{equation*}
for $s\geq 2$ and $\la\geq \la_0$, where the positive constant $C$ depends only on $\de_0$, $\al$, $\be$, and $T$. Together with \eqref{03.26.6}, we obtain that for $0<\be<\f{1}{8}(2-\al)^2$, and for all $s\geq 2$ and $\la\ge \la_0$,  
	\begin{equation*}
		\begin{split}
			&s\la\iint_{Q_0}\si \left[(\eta_t)^2+A\nabla\eta \cdot \nabla\eta\right] \df z\df t+s^3\la^3\iint_{Q_0} \si^3\eta^2\df z\df t\\
				&\leq Cs\la\iint_{\Ga_2^1\ts (0,T)} \si(\pt_r\eta)^2\df S\df t+C\iint_{Q}f^2\df z\df t\\
			&\hspace{6mm}+C\iint_{\om\ts (0,T)} \left(\vp^2+A\nabla\vp\cdot \nabla\vp\right)e^{2s\si}\df z\df t, 
		\end{split}
	\end{equation*}
	where the positive constant $C$ depends only on $\de_0, \al, \be$ and $T$. Note that 
	\begin{equation*}
		\begin{split}
			\psi_t&=e^{-s\si}\left[\eta_t+2\be s\la (t-t_0)\eta\right],\\
			\nabla\psi&=e^{-s\si}\left[\nabla\eta-s\la \eta \si (2\theta, (2-\al)r^{1-\al})\right], 
		\end{split}
	\end{equation*}
Since $\psi=\zeta \vp$, we get, for $0<\be<\f{1}{8}(2-\al)^2$ and for all $s\geq 2$ and $\la\geq \la_0$, 
\begin{equation}\label{04.05.1}
	\begin{split}
		&s\la\iint_{(3\de_0, 1-3\de_0)\ts (0,1)\ts (0,T)}\si \left[(\psi_t)^2+A\nabla\psi \cdot \nabla\psi\right] e^{2s\si}\df z\df t\\
		&\hspace{4.5mm}+s^3\la^3\iint_{(3\de_0, 1-3\de_0)\ts (0,1)\ts (0,T)} \si^3\psi^2e^{2s\si}\df z\df t\\
			&\leq Cs\la\iint_{\Ga_{2,\de_0}^1\ts (0,T)} \si(\pt_r\vp)^2\df S\df t+C\iint_Q f^2\df z\df t\\
		&\hspace{4.5mm}+C\iint_{\om\ts (0,T)} \left(s^2\vp^2+A\nabla\vp\cdot \nabla\vp\right)e^{2s\si}\df z\df t, 
	\end{split}
\end{equation}
where the positive constant $C$ depends only on $\de_0, \al, \be$ and $T$.

\medskip\noindent{\it Step 11: Removal of the auxiliary terminal condition.} We now relax \eqref{03.02.4}. 

Denote $T_0=4\de_0^{-\f{1}{2}}$. Choose $T>T_0$ and $\ga>0$ such that 
\begin{equation*}
	\de_0T^2>8+4\ga, \quad \be T^2>8+4\ga, 
\end{equation*}
Then one verifies that
\begin{equation*}
	\xi(z,0)<-\ga,\quad \xi(z,T)<-\ga, \ \fa z\in \Om,\quad \xi\left(z,\f{T}{2}\right)\geq 0, \ \fa z\in \Om, 
\end{equation*}
by taking $t_0=\f{T}{2}$. Hence, there exist $\e\in (0,\f{T}{16})$ and $\wh\ga\in (0, \f{1}{2}\ga)$ such that 
\begin{equation*}
	\begin{split}
		\xi(z,t)
		&\leq -2\wh \ga, \ \fa z\in \Om, \fa t\in (0,2\e)\cup (T-2\e, T), \\
		\xi(z,t)
		&\geq -\wh \ga, \ \fa z\in \Om, \fa \left|t-\f{T}{2}\right|\leq \e. 
	\end{split}
\end{equation*}

Let $\vp$ be the solution of \eqref{01.10.1} with respect to $(\vp^0,\vp^1,f)$, where $\vp^0\in D(\mcA), \vp^1\in H_0^1(\Om;w)$ and $f\in H^1(0,T; L^2(\Om))$. We introduce a cut-off function $k\in C^\iy(\R)$ such that 
\begin{equation*}
	0\leq k\leq 1, \quad k=1 \mbox{ on }(2\e, T-2\e), \quad k=0 \mbox{ on }(-\iy, \e)\cup (T-\e,+\iy). 
\end{equation*}
Set 
\begin{equation*}
	\psi(z,t)=k(t)\zeta(z)\vp(z,t), \ (z,t)\in Q. 
\end{equation*}
Then $\psi$ satisfies the following equation 
\begin{equation*}
	\begin{cases}
		\pt_{tt}\psi-\Div(A\nabla \psi)=h, &\mbox{in }Q_0, \\
			\psi=0, &\mbox{on }\pt\Om_0\ts(0,T),\\
		\psi(t)=\pt_t\psi(t)=0, &\mbox{at } t=0, T, 
	\end{cases}
\end{equation*}
where 
\begin{equation*}
	h=k\zeta f+2\zeta (\pt_tk)(\pt_t\vp)+\zeta \vp \pt_{tt}k-2k\nabla\zeta \cdot A\nabla\vp-k\vp\Div(A\nabla\zeta). 
\end{equation*}
Combining this with \eqref{03.26.2}, \eqref{04.05.1}, and $\Om\setminus\om=(4\de_0,1-4\de_0)\ts (0,1)$, we obtain
\begin{equation}\label{04.05.2}
	\begin{split}
		&s\la\iint_{(\Om-\om)\ts (0,T)}\si \left[(\psi_t)^2+A\nabla\psi \cdot \nabla\psi\right] e^{2s\si}\df z\df t+s^3\la^3\iint_{(\Om-\om)\ts (0,T)} \si^3\psi^2e^{2s\si}\df z\df t\\
			&\leq Cs\la\iint_{\Ga_{2,\de_0}^1\ts (0,T)} \si(\pt_r\vp)^2\df S\df t+C\iint_{\om\ts (0,T)}\left(k_t\vp_t+k_{tt}\vp\right)^2e^{2s\si}\df z\df t\\
		&\hspace{4.5mm}+C\iint_Q f^2\df z\df t+C\iint_{\om\ts (0,T)} \left(s^2\vp^2+A\nabla\vp\cdot \nabla\vp\right)e^{2s\si}\df z\df t, 
	\end{split}
\end{equation}
where the positive constant $C$ depends only on $\de_0, \al, \be$ and $T$.  

Denote the constants
\begin{equation*}
	A_0\equiv e^{-\la \wh\ga},\quad A_1=e^{-2\la \wh \ga}, 
\end{equation*}
Then $A_0\geq A_1$, and 
\begin{equation*}
	\begin{split}
		e^{2A_0s}s\int_{\f{T}{2}-\e}^{\f{T}{2}+\e}E(t)\df t
			&\leq Ce^{Cs}s\iint_{\Ga_{2,\de_0}^1\ts (0,T)}(\pt_r\vp)^2\df S\df t+Ce^{Cs}\iint_Q f^2\df z\df t\\
		&\hspace{4.5mm}+
		Ce^{2A_1s}\iint_Q \left[(\vp_t)^2+A\nabla\vp\cdot \nabla\vp\right]\df z\df t\\
		&\hspace{4.5mm}+Ce^{Cs}\iint_{\om\ts (0,T)}\left[(\vp_t)^2+A\nabla\vp\cdot \nabla\vp+s^2\vp^2\right]\df z\df t, 
	\end{split}
\end{equation*}
where 
	\begin{equation*}
		E(t)=\f{1}{2}\int_{\Om}\left[(\vp_t)^2+A\nabla\vp\cdot\nabla\vp\right]\df z.
	\end{equation*}
Therefore,
\begin{equation}\label{04.05.3}
	\begin{split}
		e^{2A_0s}s\int_{\f{T}{2}-\e}^{\f{T}{2}+\e}E(t)\df t
			&\leq Ce^{Cs}s\iint_{\Ga_{2,\de_0}^1\ts (0,T)}(\pt_r\vp)^2\df S\df t+Ce^{Cs}\iint_Q f^2\df z\df t\\
		&\hspace{4.5mm}+Ce^{2A_1s}E(0)T\\
		&\hspace{4.5mm}+Ce^{Cs}\iint_{\om\ts (0,T)}\left[(\vp_t)^2+A\nabla\vp\cdot \nabla\vp+s^2\vp^2\right]\df z\df t.
	\end{split}
\end{equation}

\medskip\noindent{\it Step 12: Conclusion for the homogeneous equation.} Let $f=0$. 

Since $\pt_t\vp\in L^2(0,T; H_0^1(\Om;w))$ and 
\begin{equation*}
	\f{1}{2}\int_{\Om} (\pt_t\vp)^2\df z\bigg|_{t=0}^{t=T}+\f{1}{2}\int_\Om A\nabla\vp\cdot \nabla\vp\df z\bigg|_{t=0}^{t=T}=0
\end{equation*}
by multiplying both sides of \eqref{01.10.1} by $\pt_t\vp$ and integrating by parts. This implies that $E(0)=E(t)$ for all $t\in[0,T]$. 

From \eqref{04.05.3}, we get
\begin{equation*}
	\begin{split}
		2\e e^{2A_0s}s E(0)
			&\leq Ce^{Cs}s\iint_{\Ga_{2,\de_0}^1\ts (0,T)}(\pt_r\vp)^2\df S\df t+Ce^{2A_1s}E(0)T\\
		&\hspace{4.5mm}+Ce^{Cs}\iint_{\om\ts (0,T)}\left[(\vp_t)^2+A\nabla\vp\cdot \nabla\vp+s^2\vp^2\right]\df z\df t.
	\end{split}
\end{equation*}
Therefore, 
\begin{equation*}
	\begin{split}
		E(0)
			&\leq Ce^{Cs}\iint_{\Ga_{2,\de_0}^1\ts (0,T)}(\pt_r\vp)^2\df S\df t+Ce^{-2(A_0-A_1)s}E(0)\\
		&\hspace{4.5mm}+Ce^{Cs}\iint_{\om\ts (0,T)} \left[(\vp_t)^2+A\nabla\vp\cdot \nabla\vp+s^2\vp^2\right]\df z\df t. 
	\end{split}
\end{equation*}
Taking $s\geq s_0$ with $s_0\equiv s_0(\de_0, \al, \be, T)\geq 2$  such that 
\begin{equation*}
	Ce^{-2(A_0-A_1)s}\leq \f{1}{2}, 
\end{equation*} 
we get
\begin{equation*}
	\begin{split}
			E(0)\leq C\iint_{\Ga_{2,\de_0}^1\ts (0,T)}(\pt_r\vp)^2\df S\df t+C\iint_{\om\ts (0,T)} \left[(\vp_t)^2+A\nabla\vp\cdot \nabla\vp+\vp^2\right]\df z\df t. 
	\end{split}
\end{equation*}

\begin{theorem}\label{04.05.T1}
	Let $\vp^0\in D(\mcA)$ and $\vp^1\in H_0^1(\Om;w)$. Then there exists a positive constant $C$, depending only on $\de_0,\al,\be$ and $T$, such that the solution $\vp$ of \eqref{01.10.1} satisfies 
	\begin{equation*}
				E(0)\leq C\iint_{\Ga_{2,\de_0}^1\ts (0,T)}(\pt_r\vp)^2\df S\df t+C\iint_{\om\ts (0,T)}\left[(\vp_t)^2+A\nabla\vp\cdot \nabla\vp +\vp^2\right]\df z\df t
	\end{equation*}
	provided that $T$ satisfies the sufficient condition in Remark \ref{04.05.R2}. 
\end{theorem}

\begin{theorem}\label{04.05.T2}
	Let $\vp^0\in H_0^1(\Om;w)$ and $\vp^1\in L^2(\Om)$. Then there exists a positive constant $C$, depending only on $\de_0,\al,\be$ and $T$, such that the solution $\vp$ of \eqref{01.10.1} satisfies 
	\begin{equation*}
				E(0)\leq C\iint_{\Ga_{2,\de_0}^1\ts (0,T)}(\pt_r\vp)^2\df S\df t+C\iint_{\om\ts (0,T)}\left[(\vp_t)^2+A\nabla\vp\cdot \nabla\vp +\vp^2\right]\df z\df t
	\end{equation*}
	provided that $T$ satisfies the sufficient condition in Remark \ref{04.05.R2}. 
\end{theorem}

\begin{remark}\label{04.05.R1}
		The lower-order interior term in Theorems \ref{04.05.T1} and \ref{04.05.T2} is a natural byproduct of the present Carleman-multiplier argument. In strictly hyperbolic settings, such a term is often removed by a compactness-uniqueness argument; see, for example, \cite{Lions,Weiss,Zuazua}. For the present model, however, this would require an appropriate global unique continuation property across the degenerate boundary set $\Ga_2^0$ for a multi-dimensional degenerate hyperbolic operator. The usual strictly hyperbolic microlocal and Holmgren-type arguments do not apply directly at $r=0$, where the principal symbol loses uniform ellipticity in the normal direction and the natural boundary topology changes with the degeneracy.

		The separated structure of the rectangle also shows that this issue is not merely technical. Expanding in the sine basis in $\theta$ reduces the homogeneous equation to the family
		\[
			\pt_{tt}\Phi_n-\pt_r(r^\al \pt_r\Phi_n)+(n\pi)^2\Phi_n=0,\qquad n\geq 1.
		\]
		Let $R_1$ be a first radial eigenfunction of $-\pt_r(r^\al\pt_r)$ with the Dirichlet condition at $r=0,1$, and consider the high tangential modes
		\[
			\vp_n(\theta,r,t)=R_1(r)\sin(n\pi\theta)\cos(\omega_n t),\qquad 
			\omega_n^2=(n\pi)^2+\rho_1 .
		\]
		Their energy grows like $(n\pi)^2$, whereas the top-boundary trace $\pt_r\vp_n|_{\Ga_2^1}$ is governed by $R_1'(1)$ and does not grow with $n$. Hence a pure estimate using only the normal trace on the top boundary cannot be expected uniformly over all tangential frequencies. Removing the interior term would therefore require additional information not used in the proof below, for instance an enlarged observation set, restrictions on the admissible spectral content, a UCP adapted to the weighted principal symbol, or a separated-variable spectral theory with uniform constants under additional hypotheses. Accordingly, the estimate proved above should be understood as the quantitative bound furnished by the current direct energy framework.
\end{remark}

\begin{remark}\label{04.05.R2}
	A sufficient condition ensuring that the proof works is
	\[
		T>\max\left\{4\de_0^{-1/2},\sqrt{\frac{8}{\be}}\right\}.
	\]
	In fact, this guarantees the existence of $\ga>0$ such that $\de_0T^2>8+4\ga$ and $\be T^2>8+4\ga$, which is exactly the requirement used in Step 11.
\end{remark}

\begin{proof}
	Choose $\vp_n^0\in C_0^\iy(\Om)$ and $\vp_n^1\in C_0^\iy(\Om)$ such that 
	\begin{equation*}
			\vp_n^0\ra \vp^0 \mbox{ strongly in } H_0^1(\Om;w), \mbox{ and } \vp_n^1\ra \vp^1\mbox{ strongly in } L^2(\Om) 
	\end{equation*}
	as $n\ra\iy$. Let $\vp_n$ be the solution of \eqref{01.10.1} corresponding to $(\vp_n^0,\vp_n^1)$, and let $\vp$ be the solution corresponding to $(\vp^0,\vp^1)$.
	
	For $n,m\in\N$, set $w_{n,m}=\vp_n-\vp_m$. Then $w_{n,m}$ solves \eqref{01.10.1} with $f=0$ and initial data
	\begin{equation*}
			w_{n,m}(0)=\vp_n^0-\vp_m^0,\qquad \pt_t w_{n,m}(0)=\vp_n^1-\vp_m^1.
	\end{equation*}
	Applying \eqref{03.12.1} to $w_{n,m}$, we obtain
	\begin{equation*}
		\begin{split}
			&\esssup_{t\in [0,T]}\left(\|w_{n,m}(t)\|_{H_0^1(\Om;w)}+\|\pt_t w_{n,m}(t)\|_{L^2(\Om)}\right)+\|\pt_r w_{n,m}\|_{L^2(0,T;L^2(\Ga_2^1))}\\
			&\leq C\left(\|\vp_n^0-\vp_m^0\|_{H_0^1(\Om;w)}+\|\vp_n^1-\vp_m^1\|_{L^2(\Om)}\right).
		\end{split}
	\end{equation*}
	Hence $\{\vp_n\}$ is Cauchy in
	\[
		L^\iy(0,T;H_0^1(\Om;w))\cap W^{1,\iy}(0,T;L^2(\Om)),
	\]
	and $\{\pt_r\vp_n\}$ is Cauchy in $L^2(0,T;L^2(\Ga_2^1))$. By the well-posedness result in Theorem \ref{03.11.T1}, the limits are necessarily $\vp$ and $\pt_r\vp$, respectively. In particular,
	\begin{equation*}
		\vp_n\ra \vp \mbox{ strongly in } L^2(0,T;H_0^1(\Om;w)),\quad \pt_t\vp_n\ra \pt_t\vp \mbox{ strongly in } L^2(Q),
	\end{equation*}
	and
	\begin{equation*}
		\pt_r\vp_n\ra \pt_r\vp \mbox{ strongly in } L^2(0,T;L^2(\Ga_2^1)).
	\end{equation*}
	In particular, the quadratic interior terms converge as well, because $\pt_t\vp_n\to \pt_t\vp$ in $L^2(Q)$, $A^{1/2}\nabla\vp_n\to A^{1/2}\nabla\vp$ in $L^2(Q)$, and $\vp_n\to \vp$ in $L^2(Q)$.
		Therefore, Theorem \ref{04.05.T1} yields, for every $T$ satisfying Remark \ref{04.05.R2},
	\begin{equation*}
				E_n(0)\leq C\iint_{\Ga_{2,\de_0}^1\ts (0,T)}(\pt_r\vp_n)^2\df S\df t+C\iint_{\om\ts (0,T)}\left[(\pt_t\vp_n)^2+A\nabla\vp_n\cdot \nabla\vp_n +\vp_n^2\right]\df z\df t, 
	\end{equation*}
	where the positive constant $C$ depends only on $\de_0, \al, \be$ and $T$. Letting $n\ra\iy$, we obtain the desired inequality for $\vp$.
\end{proof}

\section{Framework-Level Obstruction}\label{S4}

The purpose of this section is not to develop an endpoint theory at $\al=1$, nor to claim that observability fails at the critical exponent. Instead, we isolate a simpler obstruction: the weighted Dirichlet coercivity underlying the subcritical framework of Section \ref{S2} cannot remain uniform at the critical threshold. In this sense, the analysis carried out in the present paper is intrinsically subcritical.

For $\de\in (0,1)$, denote 
\[
\Om_\de=(0,1)\ts (\de,1). 
\]

\begin{proposition}\label{04.10.P1}
	Let $C_\de^{\rm crit}$ be the best constant in the critical truncated Hardy-Poincar\'e inequality
	\begin{equation}\label{04.10.1}
		\iint_{\Om_\de}\f{u^2}{r}\df z\leq C_\de^{\rm crit}\iint_{\Om_\de} r(\pt_r u)^2\df z 
	\end{equation}
	for all $u\in H^1(\Om_\de)$ satisfying $u(\theta,1)=0$. Then
	\begin{equation}\label{04.10.2}
		C_\de^{\rm crit}=\f{4}{\pi^2}|\ln \de|^2.
	\end{equation}
	In particular, the critical coercivity constant blows up at the logarithmic rate $|\ln \de|^2$ as $\de\ra 0^+$.
\end{proposition}

\begin{proof}
	For a.e.~$\theta\in (0,1)$, the slice $u_\theta(r):=u(\theta,r)$ belongs to $H^1(\de,1)$ and satisfies $u_\theta(1)=0$ in the trace sense. Let
	\[
		x=-\ln r,\qquad L=|\ln \de|,\qquad v_\theta(x)=u_\theta(e^{-x}).
	\]
	Then $r\in (\de,1)$ corresponds to $x\in (0,L)$, and a direct computation yields
	\[
		\int_\de^1 \f{u_\theta(r)^2}{r}\df r=\int_0^L v_\theta(x)^2\df x,\qquad \int_\de^1 r(\pt_r u_\theta(r))^2\df r=\int_0^L (\pt_x v_\theta(x))^2\df x.
	\]
	Hence, for each such $\theta$, \eqref{04.10.1} is reduced to the standard one-dimensional Poincar\'e inequality on $(0,L)$ with the endpoint condition $v_\theta(0)=0$.
	
	The optimal constant is therefore the reciprocal of the first eigenvalue of
	\[
		-\pt_{xx}v=\mu v,\qquad v(0)=0,\qquad \pt_x v(L)=0.
	\]
	Therefore
	\[
		\mu_1=\f{\pi^2}{4L^2},\qquad C_\de^{\rm crit}=\f{1}{\mu_1}=\f{4}{\pi^2}L^2=\f{4}{\pi^2}|\ln \de|^2.
	\]
	Integrating the one-dimensional inequality with respect to $\theta\in (0,1)$ gives the corresponding two-dimensional statement on $\Om_\de$.
\end{proof}

\begin{remark}\label{04.10.R1}
		Proposition \ref{04.10.P1} does not prove that the endpoint problem at $\al=1$ is ill-posed, nor does it prove that observability fails at $\al=1$. What it shows is narrower and more robust: the coercive mechanism supporting the present weighted Dirichlet framework cannot remain uniform at the critical threshold. Thus the obstruction is functional analytic rather than a statement about all possible endpoint formulations. A treatment of $\al=1$ would require a critical, typically logarithmic, functional framework adapted to the endpoint topology.
\end{remark}

\begin{remark}\label{04.10.R2}
	The subcritical analysis in Section \ref{S2} relies on the Hardy inequality in Lemma \ref{03.11.L1} and on the resulting norm equivalence in Remark \ref{03.11.R1}. In particular, the weighted $L^2$ term is controlled by the radial energy through a coercive constant of order $(1-\al)^{-2}$. Proposition \ref{04.10.P1} shows that, at the critical exponent $\al=1$, the analogous coercive constant on truncated domains necessarily grows like $|\ln \de|^2$. Therefore, if one insists on preserving the same weighted Dirichlet-trace framework and the same coercive norm equivalence, no uniform continuation of the present subcritical mechanism to $\al=1$ is possible.
\end{remark}

\begin{proposition}\label{04.10.P2}
	Let $\widetilde C_\de^{\rm crit}$ be the best constant in
	\begin{equation}\label{04.10.3}
		\iint_{\Om_\de}\f{u^2}{r}\df z\leq \widetilde C_\de^{\rm crit}\iint_{\Om_\de} r(\pt_r u)^2\df z
	\end{equation}
	for all $u\in H_0^1(\Om_\de)$. Then
	\begin{equation}\label{04.10.4}
		\widetilde C_\de^{\rm crit}=\f{1}{\pi^2}|\ln \de|^2.
	\end{equation}
\end{proposition}

\begin{proof}
	By density, it is enough to consider $u\in C_0^\iy(\Om_\de)$. For each fixed $\theta\in (0,1)$, the slice $u_\theta(r)=u(\theta,r)$ vanishes at both $r=\de$ and $r=1$. With the same logarithmic change of variables
	\[
		x=-\ln r,\qquad L=|\ln \de|,\qquad v_\theta(x)=u_\theta(e^{-x}),
	\]
	we obtain
	\[
		\int_\de^1 \f{u_\theta(r)^2}{r}\df r=\int_0^L v_\theta(x)^2\df x,\qquad \int_\de^1 r(\pt_r u_\theta(r))^2\df r=\int_0^L (\pt_x v_\theta(x))^2\df x,
	\]
	and now $v_\theta(0)=v_\theta(L)=0$. Hence \eqref{04.10.3} reduces to the one-dimensional Dirichlet-Dirichlet Poincar\'e inequality on $(0,L)$, whose best constant is the reciprocal of the first eigenvalue of
	\[
		-\pt_{xx}v=\mu v,\qquad v(0)=v(L)=0.
	\]
	Since $\mu_1=\pi^2/L^2$, we obtain $\widetilde C_\de^{\rm crit}\leq L^2/\pi^2$.

	To see that this constant is optimal, choose $u(\theta,r)=\sin(\pi\theta)\sin(\pi x/L)$ with $x=-\ln r$. This function belongs to $H_0^1(\Om_\de)$ and saturates the one-dimensional radial Poincar\'e quotient after integration in $\theta$. Therefore $\widetilde C_\de^{\rm crit}=L^2/\pi^2=\pi^{-2}|\ln\de|^2$.
\end{proof}

\begin{remark}\label{04.10.R3}
	Proposition \ref{04.10.P2} is the exact analogue of Proposition \ref{04.10.P1} for the fully truncated Dirichlet space. It shows that the same logarithmic obstruction persists even when the truncation is matched more directly with the Dirichlet framework used in Section \ref{S2}; only the sharp constant changes from $4\pi^{-2}$ to $\pi^{-2}$.
\end{remark}

\end{document}